\documentclass[10pt]{article}
\usepackage{amsfonts}
\usepackage{amssymb}
\usepackage{amsmath,amsthm}
\usepackage{amsmath}    
\usepackage{bbm}
\usepackage{graphicx}
\usepackage{latexsym}
\usepackage{mathrsfs}
\usepackage[all]{xy}
\usepackage{hyperref}

\usepackage{geometry}
\newcommand{\wa}{\hat}

\newcommand{\cA}{\mathcal{A}}
\newcommand{\cB}{\mathcal{B}}

\newcommand{\cE}{\mathcal{E}}
\newcommand{\cF}{\mathcal{F}}

\newcommand{\cH}{\mathcal{H}}

\newcommand{\cL}{\mathcal{L}}

\newcommand{\cR}{\mathcal{R}}

\newcommand{\cBH}{\mathcal{BH}}

\newcommand{\rC}{\mathrm{C}}

\newcommand{\sA}{\mathscr{A}}
\newcommand{\sB}{\mathscr{B}}

\newcommand{\sE}{\mathscr{E}}
\newcommand{\sF}{\mathscr{F}}

\newcommand{\sH}{\mathscr{H}}

\newcommand{\Si}{\Sigma} % Insieme simpliciale
\newcommand{\bC}{\mathbb{C}}
\newcommand{\bT}{\mathbb{T}}
\newcommand{\bM}{\mathbb{M}}
\newcommand{\bN}{\mathbb{N}}
\newcommand{\bR}{\mathbb{R}} % Insieme numeri reali
\newcommand{\bU}{\mathbb{U}}
\newcommand{\bZ}{\mathbb{Z}}

\newcommand{\si}{\sigma}

\newcommand{\eps}{\varepsilon}
\newcommand{\e}{\epsilon}

\DeclareMathAlphabet\EuFrak{U}{euf}{m}{n}	%  the Bold Euler Fraktur
\SetMathAlphabet\EuFrak{bold}{U}{euf}{b}{n}	%  gothic font

\newcommand {\ovl}{\overline}

\newcommand{\ad}{\mathrm{ad}}

\author{\textsc{Giuseppe Ruzzi$^{1}$ and Ezio Vasselli$^{2}$}
\footnote{Both the authors are supported by the  EU network ``Noncommutative Geometry" MRTN-CT-2006-0031962.}\\
  \null\\
\small{$^{1}$Dipartimento di Matematica, Universit\`a di Roma ``Tor Vergata'',}\\
\small{Via della Ricerca Scientifica, I-00133 Roma,  Italy.}  \\
\small{\texttt{ruzzi@mat.uniroma2.it}} \\[5pt]
\small{$^{2}$Dipartimento di Matematica, Universit\`a di Roma ``La Sapienza'',}\\
\small{Piazzale Aldo Moro 5, I-00185 Roma, Italy.}\\
 \small{\texttt{ ezio.vasselli@gmail.com  }}\\[20pt]
}
        
\date{}

\title{\textsc{The $K$-homology of nets of $C^*$-algebras}}

\begin{document}
\maketitle

\begin{abstract}
Let $X$ be a space, intended as a possibly curved space-time,
and $\cA$ a precosheaf of $C^*$-algebras on $X$.
Motivated by algebraic quantum field theory,
we study the Kasparov and $\Theta$-summable $K$-homology of $\cA$
interpreting them in terms of the holonomy equivariant $K$-homology 
of the associated $C^*$-dynamical system.
This yields a characteristic class for $K$-homology cycles of $\cA$
with values in the odd cohomology of $X$, 
that we interpret as a generalized statistical dimension.
\end{abstract}

\tableofcontents
\markboth{Contents}{Contents}
\newpage

%%*********************************************************************

  \theoremstyle{plain}
  \newtheorem{definition}{Definition}[section]
  \newtheorem{theorem}[definition]{Theorem}
  \newtheorem{proposition}[definition]{Proposition}
  \newtheorem{corollary}[definition]{Corollary}
  \newtheorem{lemma}[definition]{Lemma}

  \theoremstyle{definition}
  \newtheorem{remark}[definition]{Remark}
    \newtheorem{example}[definition]{Example}

\theoremstyle{definition}
  \newtheorem{ass}{\underline{\textit{Assumption}}}[section]

%*************************************************************
% \tableofcontents
% \markboth{Contents}{Contents}
%*************************************************************

\numberwithin{equation}{section}

\section{Introduction.}

Let $X$ be a topological space with base $K$.
A \emph{precosheaf of $\rC^*$-algebras} over $K$ 
is given by a pair $(\cA,\jmath)_K$, where $\cA = \{ \cA_Y \}_{Y \in K}$ is a family of $\rC^*$-algebras
and $\jmath \ $ is a family of *-monomorphisms
\[
\jmath_{Y'Y} : \cA_Y \to \cA_{Y'}
\ \ : \ \
\jmath_{Y''Y} \ = \ \jmath_{Y''Y'} \circ \jmath_{Y'Y} \ \ , \ \ Y \subseteq Y' \subseteq Y'' \ .
\]
These objects have been intensively studied during the last fifty years
in the setting of algebraic quantum field theory:
the fundamental idea is that, for any $Y \in K$, $\cA_Y$ is the $\rC^*$-algebra of quantum observables localized within $Y$.\\
\indent In the initial approach of Haag, Kastler and Araki $X$ is the Minkowski space-time
and $K$ is upward directed under inclusion. This is a technically relevant property,
as it implies that there is a Hilbert space $H$ such that
$\cA$ can be realized as a family of $C^*$-subalgebras of $B(H)$
(in technical terms, we say that $\cA$ has a faithful \emph{Hilbert space representation}).\\
\indent Now, in general relativity and conformal theory we have space-times where the base of interest
is not upward directed, and non-trivial interactions between the topology of $X$ and the geometry of $(\cA,\jmath)_K$ appear.
At the mathematical level, an interesting point is that in general it is not possible to represent 
a precosheaf on a fixed Hilbert space (\cite{RV12}), 
whilst the more general notion of representation twisted by a unitary cocycle 
(that in the sequel we simply call \emph{representation})
works well both in terms of existence results (\cite{RV12,RV12a})
and in describing generalized quantum charges (\cite{BR09,BFM09,Vas14}).\smallskip

The present paper is addressed to the study of the $K$-homology of $(\cA,\jmath)_K$,
on the ground of the above notion of representation.\smallskip

Actually Fredholm modules on precosheaves of $\rC^*$-algebras 
have already been object of study in (algebraic) quantum field theory,
both at the level of $K$-homology (\cite[\S 6]{Lon01}) 
and the one of $\theta$-summable spectral triples (see \cite[Chap. 4.9.$\beta$]{Con}, \cite{JLO88,CHKL10} and related references),
with the motivation that the supercharge operator (an odd square root of the Hamiltonian)
has a physically interesting index. 
In the above-cited references Fredholm modules are based on \emph{Hilbert space} representations
thus, on the light of the previous considerations, it is of interest to study the generalized 
notion of \emph{Fredholm $(\cA,\jmath)_K$-module}, 
based on the use of representations in the above sense.

A crucial point for our approach is that any $\rC^*$-precosheaf $(\cA,\jmath)_K$ defines a 
$\rC^*$-dynamical system for the fundamental group of $X$,
\[
\Pi := \pi_1(X) \ \ , \ \ \alpha : \Pi \to {\bf aut}A \ ,
\]
that we call the \emph{holonomy dynamical system of $(\cA,\jmath)_K$}, 
having the property that covariant representations
of $(A,\alpha)$ are in one-to-one correspondence with representations of $(\cA,\jmath)_K$
{\footnote{
In particular, Hilbert space representations of $(\cA,\jmath)_K$ correspond to \emph{invariant} representations of $(A,\alpha)$,
that are those of the type
$\pi : A \to B(H)$ with $\pi \circ \alpha_g = \pi$, $\forall g \in \Pi$.}}.
Our main result (Theorem \ref{thm.kg}) shows that Fredholm $(\cA,\jmath)_K$-modules 
are classified by the equivariant $K$-homology $K^0_\Pi(A)$ and that,
in particular, any cycle in $K^0_\Pi(A)$ arises from a Fredholm $(\cA,\jmath)_K$-module.

As a consequence of our result we obtain an explicit link between the $\rC^*$-precosheaf structure and the geometry of the underlying space.
In fact, studying the index map of Fredholm $(\cA,\jmath)_K$-modules we obtain $K$-theory classes of locally constant vector bundles,
and this allows us to define a characteristic class
\[
CCS : \sF(\cA,\jmath)_K \to \bZ \oplus H^{odd}(X,{\mathbb{R/Q}}) \ ,
\]
where 
$\sF(\cA,\jmath)_K$ is the set of Fredholm $(\cA,\jmath)_K$-modules
and
$H^{odd} := \oplus_k H^{2k+1}$ is the odd cohomology of $X$ (\S \ref{B2}).
The feature that the odd cohomology appears instead of the even one is due to the locally constant structure
of the involved vector bundles, which, as well-known, are better described by the Cheeger-Chern-Simons character rather than the usual Chern character.
In cases of physical interest (as the conformal sphere, order two de Sitter space-time and anti-de Sitter space-times)
we prove that $CCS$ is "surjective" (\S \ref{B3}), showing that our $K$-homology does indeed capture non-trivial invariants of $X$.
This solves a point arose in \cite{RV1}, where it is proved that
Fredholm $(\cA,\jmath)_K$-modules yield continuous families of Fredholm operators which,
nevertheless, \emph{forget} the information of the holonomy defined by the underlying 
representation of $(\cA,\jmath)_K$, from which the cohomology classes of $CCS$ are induced.

\smallskip

When $(\cA,\jmath)_K$ is a $\rC^*$-precosheaf generated by quantum fields in the Minkowski space-time and the Fredholm module is defined by a supercharge,
$CCS$ takes values in $\bZ$ and can be interpreted, as argued by Longo (\cite{Lon01}), as the statistical dimension, 
an important invariant related to the statistics of particles which in this way acquires a property of homotopy invariance 
that remains hidden in the original definition.
In \S \ref{B4} we give a variant of the argument of \cite[\S 6]{Lon01} for generic space-times $X$:
without using supercharges, whose existence and technical properties are problematic in space-times with at least two spatial dimensions,
we construct Fredholm modules associated to sectors with non-trivial holonomy in the sense of \cite{BR09,Vas14},
obtaining an index with values in $\bZ \oplus H^{odd}(X,{\mathbb{R/Q}})$.

\smallskip

Finally, we consider the notion of \emph{net of spectral triples}, generalizing,
in the sense of representations, the one in \cite{CHKL10}. As in the case of Fredholm modules,
we prove that there is an equivalence between the category of nets of spectral triples
and the one of $\Pi$-equivariant spectral triples (Theorem \ref{thm.spec3}).
Thus a net of spectral triples yields a cycle in the equivariant entire cyclic cohomology
defined in \cite{KL91,KKL91}, whose pairing with $K$-theory takes values in the ring of central functions of the fundamental group.

\smallskip

We will define in a forthcoming paper the $KK$-theory of 
a (not necessarily unital) precosheaf of $\rC^*$-algebras, 
coherent with the notion of $K$-homology that we study here (\cite{RVK}).
This will yield a new approach for the $K$-theory of precosheaves of $\rC^*$-algebras 
(already studied in \cite{CCHW,CCH}) and, in a different scenario,
for computing $K$-theoretic invariants of non-simple $\rC^*$-algebras, 
on the research line started by Kirchberg \cite{Kir}.

\section{Preliminaries.}
\label{A}

To make the present paper self-contained we recall some background material
on nets of $\rC^*$-algebras. References for further details are \cite{RV12,RRV09,Ruz05}.

\subsection{Homotopy of posets.}
\label{A0}

Although the applications we have in mind concern with partially ordered sets (\emph{posets}) arising as a base, ordered under inclusion, for the topology of a space, many of the results given in the present paper hold for an abstract  poset,  i.e.  a non-empty set $K$ endowed with an order relation $\leq$. We shall denote the elements of $K$ mainly by the Latin letters $a,o$ and variations of these. \\
\indent  An important invariant of $K$ is the simplicial set $\Si_*(K)$ described in \cite{Rob90,Ruz05,RRV09}. For our aims it is enough to describe the first two degrees of this set. The set $\Si_0(K)$ of $0$-simplices (``points") is nothing but that $K$. Concerning the set $\Si_1(K)$ of 1-simplices, 
any $|b|\in K$ and any ordered pair  $(\partial_0b,\partial_1b)$ of elements of $K$ fulfilling the relation $\partial_0b,\partial_1b\leq |b|$ define a $1$-\emph{simplex} $b$. 
We shall use the notation 
\[
b:=(|b|;\partial_0b,\partial_1b) \ , 
\]
and call $\partial_0b,\partial_1b$ and $|b|$, respectively, the 0-\emph{face}, the 1-\emph{face} and  the \emph{support} of $b$. 
The intuitive content of $b$ is that it is a segment starting at the ``point" $\partial_1b$
and ending at $\partial_0b$. The \emph{opposite} of  $b$ is the 1-simplex $\overline{b}$  having 
the same support as $b$ and inverted faces,  i.e.\ $\ovl b:=(|b|;\partial_1b,\partial_0b)$.\\
\indent A \emph{path}   from $a$ and to $o$,  written $p:a\to o$,   
is a finite ordered sequence  $p := b_n * \cdots * b_1$ of 1-simplices satisfying the relations 
\[
\partial_1b_1=a \ \ , \ \ \partial_0b_n=o \ \ , \ \  \partial_1 b_{i+1} = \partial_0 b_i \ , \ \mbox{ for }  i = 1 , \ldots , n-1  \ .
\]
The \emph{opposite} of $p$ is the path $\ovl  p: o\to a$ defined by 
$\ovl p := \ovl b_1 * \ldots * \ovl b_n$.
A \emph{loop over $a$} is any path of the type $p:a\to a$.  A poset $K$ is said to be \emph{pathwise connected} 
if for any pair $a,o\in K$ there is a path $p:a\to o$. Since in the present paper we shall deal only with pathwise connected posets,  \emph{from now on 
we assume that $K$ is pathwise connected}.\\   
\indent There is an equivalence relation on the set of paths of $K$ which is  
called \emph{homotopy} and written $p \sim q$. The quotient
\[
\pi_1^a(K) \ := \ \{ p : a \to a \} / \sim \ , \qquad a\in K \ ,
\]
is called the \emph{homotopy group} of $K$. Due to pathwise-connectedness,
$\pi_1^a(K)$ does not depend, up to isomorphism, on the choice of the base-point $a$. This isomorphism class is the \emph{fundamental group} 
of $K$, and the poset is said to be \emph{simply connected} whenever the fundamental group is trivial.  In particular, when $K$ is a base of 
simply and arcwise connected open sets for the topology of a space $X$, that is, $K$ is a \emph{good base},
$\pi_1^a(K)$ is isomorphic to the fundamental group $\pi_1(X)$ (see \cite{Ruz05}).\\
\indent A notion  that will be frequently used in the following is that of a path-frame.  Given $a \in K$,
a \emph{path-frame over $a$}  is a choice of paths 
\[
P_a := \{ p_{oa} : a \to o \ | \ o\in K \}
\]
subjected to the condition that  $p_{aa}$ is homotopic to the trivial path $\iota_a:=(a; a,a)$ i.e.\ the 1-simplex having all the faces and the support equal to $a$. 
For convenience we denote the opposite path of $p_{oa}$ by $p_{ao}$.

\subsection{Nets of $\rC^*$-algebras and $\rC^*$-net bundles.}
\label{A1}

A \emph{precosheaf} of $\rC^*$-algebras over $K$ is given by a pair $(\cA,\jmath)_K$,
where $\cA = \{ \cA_a \}_{a \in K}$ is a family of unital $\rC^*$-algebras and 
$\jmath_{a'a} : \cA_a \to \cA_{a'}$, $a \leq a'$,
is a family of unital *-monomorphisms, \emph{the inclusion maps},  fulfilling the relations
\begin{equation}
\label{net:rel}
\jmath_{a''a} = \jmath_{a''a'} \circ \jmath_{a'a}
\  , \qquad 
a \leq a' \leq a''
\ .
\end{equation}
In the present paper we shall adopt the standard terminology in algebraic quantum field theory
and use the term \emph{net} instead of precosheaf.

A \emph{morphism} of nets $\phi : (\cA,\jmath)_K \to (\cB,y)_K$ is a family of *-morphisms
$\phi_a : \cA_a \to \cB_a$, $a \in K$,
intertwining the inclusion maps,
\begin{equation}
\label{def.netmor}
y_{a'a} \circ \phi_a \ = \ \phi_{a'} \circ \jmath_{a'a}
\ , \qquad  
a \leq a'
\ .
\end{equation}
When $\cB_a \equiv B$ and $y_{a'a} \equiv id_B$ for all $a \leq a'$ 
(that is, $(\cB,y)_K$ is the \emph{constant net bundle})
we say that $\phi$ is a morphism from $(\cA,\jmath)_K$ to the $\rC^*$-algebra $B$ and write
\[
\phi : (\cA,\jmath)_K \to B \ .
\]
Similar definitions can be given taking other categories instead of the one of $\rC^*$-algebras.
In the present paper we shall consider 
nets of Banach and Hilbert spaces, with inclusion maps defined by isometric, linear maps,
and
nets of *-algebras, with inclusion maps given by *-monomorphisms (in this case we do not have a norm).\smallskip

A net $(\cA,\jmath)_K$ is said to be a \emph{net bundle} whenever each $\jmath_{a'a}$ is a *-isomorphism; 
to economize in notation in this case we write
\[
\jmath_{aa'} \, := \, \jmath_{a'a}^{-1} \, : \, \cA_{a'} \to \cA_a \  , \qquad  a \leq a' \ .
\]
Net bundles admit an equivalent description in terms of dynamical systems. 
Since the inclusion maps of a net bundle $(\cA,\jmath)_K$ are $^*$-isomorphisms, we can extend them first to 1-simplices by defining
\begin{equation}
\label{net:simplex}
\jmath_b:= \jmath_{\partial_0b |b|}\circ \jmath_{|b|\partial_1b} \ , \qquad b\in\Si_1(K) \ , 
\end{equation}
and then to paths,
\begin{equation}
\label{net:paths}
\jmath_p  := \jmath_{b_{n}}\circ\cdots\circ \jmath_{b_2}\circ \jmath_{b_1}  \ ,\qquad p= b_n*\cdots*b_2*b_1
%
%\circ \jmath_{|b_n| \partial_0b_{n-1}} \circ \cdots \circ \jmath_{\partial_0b_2 |b_2|}\circ \jmath_{|b_2|\partial_1b_2 }   \circ \jmath_{\partial_0b_1 |b_1|}\circ \jmath_{|b_1|\partial_1b_1 } 
%\jmath_{\partial_0b_n  |b_n|} \circ \jmath_{|b_n| \partial_0b_{n-1}} \circ \cdots \circ \jmath_{\partial_0b_2 |b_2|}\circ \jmath_{|b_2|\partial_1b_2 }   \circ \jmath_{\partial_0b_1 |b_1|}\circ %\jmath_{|b_1|\partial_1b_1 } 
%
\ .
\end{equation}
Clearly,  $\jmath_p:\cA_{a}\to \cA_{o}$ is a $^*$-isomorphism  for any path  $p:a\to o$. 
Furthermore,  the map $p\mapsto \jmath_p$ preserves path composition and, by (\ref{net:rel}), factorizes through the homotopy equivalence relation.  
Thus, picking $a\in K$ and defining $\cA_* := \cA_a$ we get the action $\jmath_* : \pi_1^a(K) \to \textbf{aut}\cA_*$ defined by 
\begin{equation}
\label{def.alphap}
\jmath_{*,[p]} := \jmath_p   \ , \qquad [p]\in\pi^a_1(K) \  , 
\end{equation}
that we call the {\em holonomy action of} $(\cA,\jmath)_K$. 
The triple $(\cA_*,\pi^a_1(K),\jmath_*)$ is called the \emph{holonomy dynamical system} of the net bundle. 
If $\phi:(\cA,\jmath)_K\to (\cA',\jmath')_K$ is a morphism, then using (\ref{def.netmor}) and 
setting $\phi_*:=\phi_a$  we immediately find that $\phi_* :\cA_* \to \cA'_* $ is a $\pi^a_1(K)$-equivariant  $^*$-morphism.  
This construction is, up to isomorphism, independent of the choice of the base-point $a$.  \\ 
\indent An interesting point is that {\em every}  $\rC^*$-dynamical system 
$(A,\pi_1^a(K),\alpha)$ comes associated to a $\rC^*$-net bundle.
For, we pick a path frame $P_a := \{ p_{oa} \ | \ o\in K\}$ and set
\begin{equation}
\label{rep:2a}
g_{o'o} \ := \ [ p_{ao'} * (o';o',o) * p_{oa} ] \  \in \pi_1^a(K)
\ ,  \qquad  o\leq o'
\end{equation}
(here $(o';o',o)$ is the 1-simplex whose support is $o'$ and whose 0- and 1-faces are, respectively, $o'$ and $o$).  
Then we define the fibres as 
$\cA_{*,o} := A$ for any $o \in K$,
and the inclusion morphisms 
\[
\alpha_{*,o'o} := \alpha_{g_{o'o}} \ , \qquad o\leq o' \ .
\] 
Homotopy invariance of $\alpha$ implies that $\alpha_*$ fulfils the net relations; 
so $(\cA_*,\alpha_*)_K$ is a $\rC^*$-net bundle. Moreover, it is easily verified that
a $\pi_1^a(K)$-equivariant morphism 
$\eta : (A,\alpha) \to (A',\alpha')$
defines the morphism
$\eta_* : (\cA_*,\alpha_*)_K \to (\cA'_*,\alpha'_*)_K$, $\eta_{*,a} := \eta$, $a \in K$.
This leads to the following results (see \cite{RV12} for details).
\begin{theorem}
\label{thm.netdyn}
The following assertions hold:
\begin{itemize}
\item[(i)] The functor
\[
(\cA,\jmath)_K \mapsto (\cA_*,\pi_1^a(K),\jmath_*) \ \ , \ \ \phi \mapsto \phi_*
\]
defines an equivalence from the category of net bundles over $K$ to the one of
$\rC^*$-dynamical systems over $\pi_1^a(K)$, whose inverse, up to isomorphism, is
\[
(A,\pi_1^a(K),\alpha) \mapsto (\cA_*,\alpha_*)_K \ \ , \ \ \eta \mapsto \eta_* \ .
\]
\item[(ii)] If $K$ is simply connected then any net bundle $(\cA,\jmath)_K$ is trivial i.e.\, it is isomorphic to 
the constant net bundle $\cA_a$.
\end{itemize} 
\end{theorem}

\

A \emph{section} of a net bundle $(\cB,\jmath)_K$ is a family $T = \{ T_o \in \cB_o \}$
such that 
$\jmath_{o'o}(T_o) = T_{o'}$, $\forall o \leq o'$.
The set of sections is denoted by $S(\cB,\jmath)_K$. Using the above considerations it is easily verified
that there is a one-to-one correspondence
\begin{equation}
\label{eq.SEC}
S(\cB,\jmath)_K \to \cB_*^\jmath \ \ , \ \ T \mapsto T_* := T_a \ ,
\end{equation}
where $\cB_*^\jmath$ is the fixed-point algebra of $\cB_*$ under the $\pi_1^a(K)$-action.

\

Finally, it is clear that the notion of net bundle applies to other categories:
in particular, if $(\cH,V)$ is a net of Banach (Hilbert) spaces where each $V_{o'o}$ is invertible (unitary), 
then we say that $(\cH,V)$ is a \emph{Banach (Hilbert) net bundle}. 
It is worth observing that all the results of the present section, stated for $\rC^*$-net bundles, apply to Banach (Hilbert) net bundles.

\subsection{Representations and the enveloping net bundle.}
\label{A2}

Let $(\cH,U)_K$ denote a Hilbert net bundle;
using the adjoint action, we define the $\rC^*$-net bundle
$(\cBH,\ad U)_K$.
Given a net of $\rC^*$-algebras $(\cA,\jmath)_K$, a \emph{representation} on $(\cH,U)_K$
is by definition a morphism
\[
\pi : (\cA,\jmath)_K \to (\cBH,\ad U)_K \ .
\]
We say that $(\cA,\jmath)_K$ is injective (non-degenerate) whenever it admits a 
faithful (non-vanishing) representation. In particular, given the Hilbert space $H$, a morphism $\pi : (\cA,\jmath)_K \to B(H)$ 
is called a \emph{Hilbert space representation}: this is the notion usually used 
in algebraic quantum field theory.\smallskip

The equivalence between  $\rC^*$-net bundles and  holonomy $\rC^*$-dynamical systems induces, up to unitary equivalence, 
a bijective correspondence between representations of the net $(\cA,\jmath)_K$ 
and $\pi_1^a(K)$-covariant representations of the corresponding holonomy dynamical system: 
for, if $\pi$ is a representation of $(\cA,\jmath)_K$ on $(\cH,U)_K$ then the pair $(\pi_*, U_*)$, where 
\begin{equation}
\label{rep:1}
\pi_*:=\pi_a \ \ , \ \ U_{*,[p]}:= U_p  \ , \qquad p : a \to a \ , 
\end{equation}
is a $\pi^a_1(K)$-covariant representation of $(\cA_*,\pi^a_1(K),\jmath_*)$.
Conversely, if $(\eta,V)$ is a $\pi^a_1(K)$-covariant representation of 
$(\cA_*,\pi^a_1(K),\jmath_*)$ on a Hilbert space $H$, then using (\ref{rep:2a}) and setting 
\begin{equation}
\label{rep:2}
\eta_{*,o}:= \eta  \ \  ,  \  \   
\cH_{*,o}:= H  \ \  ,  \ \  
V_{*,\tilde o o} := V_{g_{\tilde oo}}
\ \ ,  \ \ \forall o \leq \tilde o  \ , 
\end{equation}
we get a representation $\eta_*$ of 
$(\cA_{**},\jmath_{**})_K \simeq (\cA,\jmath)_K$ 
on the Hilbert net bundle $(\cH_*,V_*)_K$.
As a consequence $\rC^*$-net bundles are injective nets:
in fact, faithful covariant  representations of  holonomy dynamical systems induce faithful representations of the associated net bundle,
and $\rC^*$-dynamical systems have faithful covariant representations. This, in turn, implies that a net is injective if, and only if, 
it admits an injective morphism into a $\rC^*$-net bundle. \\
\indent For any net of $\rC^*$-algebras $(\cA,\jmath)_K$ there exists 
a $\rC^*$-net bundle $(\bar \cA,\bar \jmath)_K$  and a morphism 
\begin{equation}
\label{eq.env}
\e : (\cA,\jmath)_K \to (\bar \cA,\bar \jmath)_K \ ,
\end{equation}
fulfilling the universal property of lifting in a unique way any morphism  from 
$(\cA,\jmath)_K$ to a $\rC^*$-net bundle  $(\cB,y)_K$, as follows:
\begin{equation}
\label{eq.env.u}
\xymatrix{
   (\bar \cA,\bar \jmath)_K
   \ar[r]^{\psi^\uparrow} 
&  (\cB,y)_K
\\ (\cA,\jmath)_K
   \ar[u]^{\e} 
   \ar[ru]_{\psi} 
&  
}
\end{equation}
$\e$  and $(\bar \cA,\bar \jmath)_K$ are  called, respectively, the \emph{canonical morphism} and the {\em enveloping $\rC^*$-net bundle} of the net.  
Assigning the enveloping net bundle is a functor: so any morphism
$\phi : (\cA,\jmath)_K \to (\cA', \jmath')_K$ 
is lifted to a morphism  
$\bar \phi : (\bar \cA,\bar \jmath)_K \to (\bar\cA',\bar \jmath')_K$ 
of the corresponding enveloping net bundles such that  $\bar \phi \circ \e = \e' \circ \bar \phi$, 
%
%\begin{equation}
%\label{eq.env.func}
%\bar \phi \circ \e = \e' \circ \bar \phi \ ,
%\end{equation}
%
where $\e'$ is the canonical morphism of $(\cA',\jmath')_K$.\\
\indent On these grounds, the \emph{holonomy $\rC^*$-dynamical system} $( \, A,\pi_1^a(K) \,)$ of $(\cA,\jmath)_K$ 
is defined as the holonomy dynamical system associated with $(\bar \cA,\bar \jmath)_K$.
By construction, covariant representations of $(A,\pi_1^a(K))$ and representations
of $(\cA,\jmath)_K$ are in one-to-one correspondence. 
An explicit description of $A$,
that is the unique, up-to-isomorphism, fibre of $(\bar \cA,\bar \jmath)_K$,
will be needed only in \S \ref{C}, see (\ref{def.spec3.1}).
Here we note that using the path evaluations (\ref{net:paths}),
and recalling the definitions (\ref{eq.env}) and (\ref{rep:1}), 
given $a \in K$ and a representation 
$\pi : (\cA,\jmath) \to (\cBH,\ad U)_K$
we have the equalities
\begin{equation}
\label{eq.circ}
\ad U_p \circ \pi_o   =   \pi_* \circ \bar{\jmath}_p \circ \e_o
\  , \qquad 
\forall o \in K \ , \ p : o \to a
\ ,
\end{equation}
relative to *-morphisms from $\cA_o$ into $B(H) = B(\cH_a)$.

When the poset $K$ is simply connected any net bundle over $K$ is trivial, i.e. it is isomorphic to the  constant net bundle.   
In this case the above construction can be recast in terms of morphisms from nets to $\rC^*$-algebras:  
there is a  canonical morphism 
$\e:(\cA,\jmath)_K\to\vec{A}$, where $\vec{A}$  is the  \emph{universal $\rC^*$-algebra} of the net
{\footnote{
When the poset is simply connected  the standard fibre of the enveloping net bundle 
is isomorphic to $\vec{A}$. This $\rC^*$-algebra has been introduced by Fredenhagen in \cite{Fre90}.}, 
satisfying the universal property  that  for any morphism  $\psi: (\cA,\jmath)_K\to B$  
there is a unique morphism $\psi^\uparrow:\vec{A}\to B$ 
such $\psi^\uparrow \circ \e= \psi$. Also the above functorial properties can be easily recast in these terms.

\section{$K$-homology and index.}
\label{B}

In the present section we prove our main results and study \emph{nets of Fredholm modules}
that yield the cycles for the $K$-homology of a net of $\rC^*$-algebras.
We give a complete classification of these objects in terms of the equivariant
$K$-homology of the holonomy $\rC^*$-dynamical system of the given net.

\subsection{Nets of Fredholm modules.}
\label{B1}

A net of $\rC^*$-algebras  $(\cA,\jmath)_K$ is said to be  \emph{graded} whenever there is a period 2 automorphism
$\beta \in {\bf aut}(\cA,\jmath)_K$, $\beta^2 = id_\cA$,
and \emph{ungraded} whenever $\beta = id_\cA$.
We shall denote graded nets by  $(\cA,\jmath;\beta)_K$. A \emph{graded} morphism $\phi:(\cA,\jmath;\beta)_K\to (\cA',\jmath';\beta')_K$ is a morphism
from $(\cA,\jmath)_K$ to  $(\cA',\jmath')_K$ intertwining the gradings, 
\begin{equation}
\label{grad.mor}
\phi_o\circ \beta_o \ = \ \beta'_o\circ \phi_o \ , \qquad o\in K \ . 
\end{equation}
The notion of graded net applies to other categories other than that of $\rC^*$-algebras. In particular, we see that a \emph{grading} for a  
Hilbert net bundle $(\cH,U)_K$ amounts to the existence of a family $\Gamma = \{ \Gamma_o \in U \cH_o\}$ of unitary operators satisfying 
$\Gamma_{o'} U_{o'o} = U_{o'o} \Gamma_o$ for any inclusion $o\leq o'$ and $\Gamma_o^2 = 1_o$ for any $o$. We observe that
the associated $\rC^*$-net bundle $(\cBH,\ad U)_K$  turns out to be graded as well by the the adjoint action  $\ad\Gamma$,
and that this grading splits  $(\cBH,\ad U)_K$  into Banach net bundles $(\cB^\pm \cH , \ad U)_K$ with fibres
\[
B^\pm(\cH_o) \ := \ \{ t\in B(\cH_o) \, | \,  \ad \Gamma_o(t) = \pm t\} \ \ , \ \ o\in K \ .
\]
On these grounds,  a \emph{graded representation}  $\pi$ of a graded net $(\cA,\jmath; \beta)_K$ into a graded 
Hilbert  net bundle $(\cH,U;\Gamma)$ is nothing but that  a graded morphism $\pi:(\cA,\jmath; \beta)_K\to  (\cBH,\ad U;\ad\Gamma)_K$.
We say that $\pi$ is \emph{ungraded} whenever both $(\cA,\jmath)_K$ and $(\cH,U)_K$ are ungraded.

\smallskip

We are now ready to introduce the main notion of this section.
\begin{definition}
\label{def.fred}
A \textbf{even/odd Fredholm module} over the graded/ungraded net $(\cA,\jmath; \beta)_K$ is a triple $(\pi,U,F)$, 
where $\pi : (\cA,\jmath; \beta)_K \to (\cBH,\ad U;\ad\Gamma)_K$ 
is an even/ungraded representation and $F = \{ F_o = F_o^* \in B(\cH_o) \}$ fulfils, in the even case, the relations
$\ad\Gamma_o(F_o) = - F_o$, $\forall o \in K$,
and, in all the cases,
\begin{equation}
\label{eq.khom}
U_{o'o} F_o = F_{o'} U_{o'o} \ \ , \ \
F_o^2-1_o \ \ , \ \ 
[F_o,\pi_o(t)] \, \in K(\cH_o)
\ \ , \ \qquad 
\forall o\leq o' \in K \ , \  t \in \cA_o \ . 
\end{equation}
We denote the set of even/odd Fredholm $(\cA,\jmath;\beta)_K$-modules by $\sF^{0/1}(\cA,\jmath)_K$.
\end{definition}
We note that the previous definition may be given also in the non-unital case, by applying the obvious modifications. 
Also note that in the even case $F$ is a section of $(\cB^- \cH , \ad U)_K$. \\
\indent Let us now pick $a \in K$. Using the previous definition we find
$\ad U_p(F_a) = F_o$ for any path $p : a \to o$ and $o\in K$,
and this implies
\begin{equation}
\label{eq.khom2}
F_a^2-1_a \ \ , \ \
\left[ F_a \ , \ \ad U_p \circ \pi_o(t) \right] \ \ , \ \
\ad U_q(F_a) - F_a 
\, \in K(\cH_a) 
\ \ , \ \qquad
\forall t\in\cA_o \ , \ q : a \to a \ .
\end{equation}
Notice that the last relation follows by the stronger equality
\begin{equation}
\label{eq.khom3}
\ad U_q(F_a) = F_a \ \ , \ \ \forall q : a \to a 
\end{equation}
(see also (\ref{eq.SEC})). 
Notice, furthermore, that in the previous equations only the evaluation of $F$ over $a\in K$ appears.
This suggests the following definition.
\begin{definition}
Let $(\cA,\jmath;\beta)_K$ be a graded/ungraded net of $\rC^*$-algebras. An even/odd \textbf{Fredholm $(\cA,\jmath;\beta)_K$-module
localized in $a \in K$} is given by a graded/ungraded representation $\pi$ on $(\cH,U;\Gamma)_K$ 
and an odd/ungraded operator 
$F_a \in B(\cH_a)$
fulfilling (\ref{eq.khom2}). We denote the set of even/odd Fredholm modules localized in $a \in K$ by 
$\sF^{0/1}_a(\cA,\jmath)_K$. 
\end{definition}
The basic idea of the previous definition is that in general we cannot coherently extend $F_a$ to a section of $(\cBH,\ad U)_K$;
in fact, if we try to define $F_o := \ad U_p(F_a)$ by means of the parallel displacement over $p : a \to o$,
then we have that $F_o$ is independent of $p$ only up to compact perturbations, unless $F_a$ is holonomy-invariant in the sense of (\ref{eq.khom3}). \\
\indent We observe that there is a one-to-one correspondence 
$\sF^*_a(\cA,\jmath)_K \leftrightarrow \sF^*_e(\cA,\jmath)_K$ for any $a,e\in K$ and $*=0,1$: 
in fact, since $K$ is pathwise connected there is a path $p : a \to e$, 
and defining $F_e := \ad U_p(F_a)$ it is easily seen that $F_e$ satisfies (\ref{eq.khom2}). 
We also point out that the evaluation over $a \in K$ yields a map
\[
\sF^*(\cA,\jmath)_K \to \sF^*_a(\cA,\jmath)_K
\ \ , \ \
*=0,1
\ ,
\]
that, as we shall see in Theorem \ref{thm.kg}, in general is not surjective.

\begin{remark}
In analogy to Baaj and Julg we define an \emph{unbounded Fredholm $(\cA,\jmath)_K$-module}
as a triple $(\pi,U,D)$, where $(\pi,U)$ is as in Def.\ref{def.fred} and $D := \{ D_o \}$ is a family
of unbounded, selfadjoint, regular
{\footnote{We recall that an operator $D : {\mathrm{Dom}}(D) \to H$ is said to be 
\emph{regular} whenever ${\mathrm{dim}} \ker D < \infty$ and ${\mathrm{Im}}(D)$ is closed.}}
operators such that
\[
\left\{
\begin{array}{ll}
\ad U_{o'o}(D_o) = D_{o'}  \ , \qquad  
  o \leq o' \in K \ ,
\\
(1+D_o^2)^{-1} \in K(\cH_o) \  , \qquad  
 o \in K \ ,
\\
{\mathrm{closure}} 
\{ t \in \cA_o : \ [D_o,\pi_o(t)] \in B(\cH_o) \} = 
\cA_o 
\ , \qquad    o \in K \ .
\end{array}
\right.
\]
As usual, $(\pi,U)$ is even/ungraded and $D$ is odd/ungraded in the even/odd case.
As in \cite[VIII.17.11]{Bla} we find that defining $F_o := D_o(1+D_o^2)^{-1}$ for any $ o \in K$
we obtain a Fredholm $(\cA,\jmath)_K$-module $(\pi,U,F)$.
Un unbounded, {\em localized} Fredholm $(\cA,\jmath)_K$-module is, by definition, a triple
$(\pi,U,D_a)$,
where $a \in K$ and $D_a$ is an odd, unbounded, regular and selfadjoint operator on $\cH_a$ 
such that 
$(\pi,U,D_a(1+D_a^2)^{-1} )$
is a (bounded) Fredholm module localized in $a$.
The fact that every bounded (localized)
Fredholm module arises from an unbounded (localized) Fredholm module shall be 
proved in Cor.\ref{cor.kg}.
\end{remark}

In the following result we show that, starting from the notion of holonomy dynamical system of a net, 
the $\pi_1^a(K)$-equivariant $K$-homology works well in describing nets of Fredholm modules, 
and that it can be completely interpreted in terms of these.
For the notions of equivariant (and strongly equivariant) $K$-homology see Appendix \ref{a.KK}.

\begin{theorem}
\label{thm.kg}
Let $K$ be a poset with fundamental group $\Pi$ and $(\cA,\jmath)_K$ a net of $\rC^*$-algebras
with holonomy dynamical system $(A,\Pi)$.
Then every localized Fredholm $(\cA,\jmath)_K$-module defines a cycle in $K^*_\Pi(A)$, 
with $*=0,1$ in the even and odd case respectively. 
Any element of $K^*_\Pi(A)$ arises in this way and, in particular, any element of the strongly equivariant 
$K$-homology $\widetilde K^*_\Pi(A)$ arises from a Fredholm $(\cA,\jmath)_K$-module.
\end{theorem}

\begin{proof}
To fix the ideas we give the proof in the even case (the odd case is analogous).
By construction there is a one-to-one correspondence between representations of $(\cA,\jmath)_K$ and 
$\Pi$-equivariant representations of $A$ (see \S \ref{A2}).
By functoriality we also find that $A$ is graded by a period 2 automorphism 
$\beta_* \in \textbf{aut}A$.\smallskip 

\emph{Step 1.} 
We prove that any localized Fredholm $(\cA,\jmath;\beta)_K$-module $(\pi,U,F_a)$
defines a cycle in $K^0_\Pi(A)$.
To this end, we consider the $\Pi$-equivariant representation (\ref{rep:1}),
\[
\pi_* : A \to B(H)
\ \ , \ \
U_* : \Pi \to U(H)
\ \ , \ \
H := \cH_a
\ , 
\]
and the grading operator defined by $\Gamma_* := \Gamma_a \in U(H)$. Since  $\Gamma$ is a section of $(\cBH,\ad U)_K$, see   (\ref{eq.SEC}), 
we conclude that $\Gamma_*$ is $\Pi$-invariant.
This implies that $U_*$ commutes with $\Gamma_*$. 
By (\ref{grad.mor}) we find $\pi_* \circ \beta_* = \ad \Gamma_* \circ \pi_*$,  so $\pi_*$ is graded.  
We define $F_* := F_a \in B(H)$ and check that it fulfils (\ref{eq.KK.khom}) (with $A$ in place of $\sA$).
The fact that $F_*$ is $\Pi$-continuous is obvious, $\Pi$ being discrete.
The relations in (\ref{eq.KK.khom}) follow by (\ref{eq.khom2}),
using the facts that: \emph{(i)}  by definition, $K(H) = K(\cH_a)$; \emph{(ii)} the $\rC^*$-algebra generated by $\{ \ad U_p \circ \pi_o(\cA_o) \}_{p \, : o \to a}$ coincides, by \cite[Eq.4.10]{RV12}, with $\pi_*(A)$. We make the explicit computation for commutativity up to $K(H)$:
\[
[ F_* \, , \, \pi_*(t) ] \ \stackrel{(ii)}{=} \ 
\sum_k [ F_a \, , \, \ad U_{p_{ao_k}} \circ \pi_{o_k}(t_k) ] \ \stackrel{ (\ref{eq.khom2}) }{\in} \
K(H)
\ .
\]
In the same way the other properties in (\ref{eq.KK.khom}) are verified,
so $(\pi_*,F_*)$ is a cycle in $K^0_\Pi(A)$ as desired.\smallskip

\emph{Step 2.}  We prove that any cycle in $K^0_\Pi(A)$ is defined by a localized Fredholm $(\cA,\jmath)_K$-module.
To this end, let $(\eta,\phi)$ denote a Kasparov $A$-$\bC$-bimodule with underlying $\Pi$-action $V$.
Then by the considerations in \S \ref{A2} there is a representation $(\pi,U)$ of $(\cA,\jmath)_K$ 
such that 
\[
\eta = \pi_a \ \ , \ \ V = U_a \ .
\]
Defining $F_a := \phi$ we obtain the localized Fredholm $(\cA,\jmath;\beta)_K$-module 
$(\pi,U,F_a)$.
In particular we check explicitly commutativity up to $K(H)$:
\[
[ F_a  , \ad U_p \circ \pi_o(t) ]  \ \stackrel{ (\ref{eq.circ}) }{=} \
[ \phi , \eta( \bar{\jmath}_p \circ \e_o(t) )    ]  \ \stackrel{ (\ref{eq.KK.khom}) }{\in} \
K(H) = K(\cH_a)
\ \ , \ \
t \in \cA_o
\ ,
\]
having considered a path
$p : o \to a$.
In the same way, the other properties of (\ref{eq.khom2}) are verified.\smallskip 

\emph{Step 3.} We consider the particular case of Fredholm $(\cA,\jmath;\beta)_K$-modules 
vs. cycles in $\widetilde K^0_\Pi(A)$.
If $(\pi,U,F)$ is a Fredholm $(\cA,\jmath;\beta)_K$-module 
then the operator $F_* := F_a \in B(\cH_a)$ is $\Pi$-invariant by (\ref{eq.SEC}), i.e.
$\ad U_p(F_a) = F_a$, for any path $p : a \to a$,
and this is equivalent to saying that $(\pi_*,F_*)$, with $\Pi$-action $U_*$, 
defines a cycle in the strongly equivariant $K$-homology $\widetilde K^0_\Pi(A)$.
Conversely, if $(\eta,\phi)$ is a Kasparov $A$-$\bC$-module with $\Pi$-action $V$ 
\emph{such that $\phi$ is $\Pi$-invariant}, then using the representation $(\pi,U)$ of the previous step we define
\[
F_o  \ :=  \ \ad U_{p_{oa}}(\phi) \in B(H) \ \ , \ \ o \in K \ ,
\]
where $P_a = \{ p_{oa} \}$ is a path frame as usual.
On the light of the previous step, to prove that $(\pi,U,F)$ is a Fredholm $(\cA,\jmath)_K$-module 
it suffices check that $F = \{ F_o \}$ is a section; to this end we compute, for all $o \leq \tilde o$,
\[
\ad U_{\tilde oo}(F_o) \ = \
\ad U_{\tilde oo} \circ \ad U_{p_{oa}} (\phi)  \ = \
\ad U_{p_{\tilde oa}} \circ \ad U_g (\phi)  \ \stackrel{ (*) }{=} \
\ad U_{p_{\tilde oa}}(\phi)  \ = \ 
F_{\tilde o} \ ,
\]
where in (*) we used $\Pi$-invariance of $\phi$ for 
$g := [ p_{a \tilde o}*(\tilde oo)*p_{oa} ] \in \Pi$. 
The other properties of Fredholm module follow in the same way as in the beginning of the proof.
\end{proof}

A remark: Fredholm $(\cA,\jmath)_K$-modules 
$\omega = (\pi,U,F)$, $\omega = (\pi',U',F')$ 
are said to be unitarily equivalent whenever there is a unitary family 
$V = \{ V_o : H_o \to H'_o \}$
such that 
\[
\Gamma'_o V_o = V_o \Gamma_o
\ \ , \ \
\pi'_o = \ad V_o \circ \pi_o
\ \ , \ \
F'_o V_o = V_o F_o
\ \ , \ \
\forall o \in K
\ .
\]
Defining $V_* := V_a : H_a \to H'_a$ we conclude, by \cite[VIII.17.2]{Bla},
that $(\pi_*,F_*)$ is homotopic to $(\pi'_*,F'_*)$, thus they define the same class in $K^*_\Pi(A)$.
%
%
%It could be of interest to express the homotopy relation between
%%
%$(\pi_*,F_*)$ and $(\pi'_*,F'_*)$
%%
%in terms of properties of $\omega , \omega'$, nevertheless this goes beyond the purpose
%of this paper, as $K^0_\Pi(A)$ already captures the informations in which we are interested
%(see the following section).

\begin{corollary}
\label{cor.kg}
For any localized Fredholm $(\cA,\jmath)_K$-module $(\pi,U,F_a)$ there is an unbounded, localized 
Fredholm $(\cA,\jmath)_K$-module defining the same class in $K^*_\Pi(A)$ as $(\pi,U,F_a)$.
\end{corollary}

\begin{proof}
By a result of Baaj-Julg (see \cite[VIII.17.11]{Bla}), for any Kasparov $A$-$\bC$-bimodule $(\eta,\phi)$ there is an unbounded 
Kasparov $A$-$\bC$-bimodule $(\eta,D)$ such that $(\eta,D(1+D^2)^{-1})$ is homotopic to $(\eta,\phi)$.
By the previous theorem, $\eta$ yields a representation $(\eta_*,U)$ of $(\cA,\jmath)_K$, 
and defining $D_a := D$ for a fixed $a \in K$ yields the desired unbounded, localized Fredholm $(\cA,\jmath)_K$-module.
\end{proof}

\subsection{On the index of nets of Fredholm modules.}
\label{B2}

Let $(\cA,\jmath)_K$ be a net of $\rC^*$-algebras and $(A,\Pi)$ denote the associated holonomy dynamical system.
Then every (eventually unbounded, or localized) even/odd Fredholm $(\cA,\jmath)_K$-module 
$\omega := (\pi,U,F_a) \in \sF^*_a(\cA,\jmath)_K$, $*=0,1$,
lives on the Hilbert net bundle $(\cH,U)_K$ associated to the representation $(\pi,U)$ of $(\cA,\jmath)_K$,
and by the results of the previous section we obtain the cycle
$\omega_{hol} := (\pi_*,F_*)$ in $K^*_\Pi(A)$,
on which we can evaluate the index map
\[
{\mathrm{index}}^{hol} : K_*^\Pi(A) \times K^*_\Pi(A)  \to R(\Pi) \ \ , \ \ *=0,1 \ ,
\]
defined by the Kasparov product (here $R(\Pi) := KK^0_\Pi(\bC,\bC)$, see \S \ref{a.KK}).

\smallskip

We now assume that $K$ is a good base generating the topology of a locally compact Hausdorff space $X$.
As mentioned in \S \ref{A0}, this implies that there is an isomorphism $\Pi \simeq \pi_1(X)$.
By \cite[Theorem 4.2]{RV1}, $(\cA,\jmath)_K$ defines a canonical $C_0(X)$-algebra $\sA$ fulfilling the following property:
if we take in particular
$\omega = (\pi,U,F) \in \sF^*(\cA,\jmath)_K$, 
then by \cite[Theorem 5.3]{RV1} a cycle 
$\omega_{top}$
is defined in the representable $K$-homology $RK^*(\sA)$ in the sense of Appendix \ref{a.KK}.
The idea is that:
\begin{itemize}
\item $(\cH,U)_K$ defines a locally constant Hilbert bundle $\sH \to X$ with associated $C_0(X)$-module of sections $\hat \sH$, see Remark \ref{rem.b0};
\item $\pi$ defines a $C_0(X)$-module representation $\pi_{top} : \sA \to B(\hat \sH)$;
\item the family $F$ yields a Fredholm $C_0(X)$-module operator $F_{top} \in B(\hat \sH)$.
\end{itemize}
Thus, we have a map
\[
\sF^*(\cA,\jmath)_K \to RK^*(\sA)
\ \ , \ \ 
\omega \mapsto \omega_{top}
\ \ , \ \ 
*=0,1
\ ,
\]
and the index pairing for $\omega_{top}$ defined by the Kasparov product reads as

\[
{\mathrm{index}}^{top} : RK_*(\sA) \times RK^*(\sA)  \to RK^0(X) \ \ , \ \ *=0,1 \ .
\]
%
%
%
%In particular, we consider the map
%%
%\begin{equation}
%\label{eq.dim}
%\tau^{hol}(\cdot) : K^*_\Pi(A)  \to R(\Pi) \ \ , \ \  \tau^{hol}(\pi_*,F_*) := {\mathrm{index}}(F_*) \ ,
%\end{equation}
%%
%that can be obtained by \emph{forgetting} the representation $\pi_*$ in the definition of Fredholm module.
%%
%%Elements of $R(\Pi)$ carry unitary
%%representations of $\Pi$ (see \S \ref{a.KK}), inducing locally constant bundles
%%of Hilbert spaces (\S \ref{a.LC}). \smallskip
%
%
In particular, by Theorem \ref{thm.kg} we know that elements of $\sF^*(\cA,\jmath)_K$
define classes in the \emph{strongly} equivariant $K$-homology $\widetilde K^*_\Pi(A)$, and it is well-known
that in this case the index takes values in $\widetilde R(\Pi)$, the ``classical" representation ring 
generated by \emph{finite dimensional} unitary representations:
\begin{equation}
\label{eq.fdim}
\tau^{hol}(\cdot) : \widetilde K^*_\Pi(A)  \to \widetilde R(\Pi) 
\ \ , \ \  
\tau^{hol}(\pi_*,F_*) := {\mathrm{index}}(F_*) \ .
\end{equation}
In the previous formula ${\mathrm{index}}(F_*)$ is obtained by \emph{forgetting} the representation $\pi_*$
and is given by the Fredholm $\Pi$-index of $F_*$ (see \cite[\S 8]{Hig} or \cite[Ex.4.2.10]{Val}).
We note that $\widetilde R(\Pi)$ can be interpreted
as the Grothendieck ring generated by \emph{finite dimensional} locally constant Hilbert bundles on $X$,
see \S \ref{a.LC}.  \smallskip

Now,  there is a loss of information when passing from ${\mathrm{index}}^{hol}$ to ${\mathrm{index}}^{top}$, namely the extra-structure 
given by the holonomy: as observed in \S \ref{a.LC} the Chern character vanishes
on classes of $RK^0(X)$ arising from locally constant vector bundles, thus it is convenient to 
consider the Cheeger-Chern-Simon character $ccs$ (see (\ref{eq.ccs})). 
Composing $ccs$ with $\tau^{hol}$, and using Thm.\ref{thm.kg}, we get the (additive) characteristic class
\begin{equation}
\label{eq.ccs.net}
CCS : \sF^*(\cA,\jmath)_K \to  \bZ \oplus H^{odd}(X,{\mathbb{R/Q}})
\ \ , \ \
*=0,1
\ ,
\end{equation}
assigning to a (possibly unbounded) Fredholm $(\cA,\jmath)_K$-module $\omega = (\pi,U,F)$ the ccs-character of 
$\tau^{hol}(\omega_{hol}) \in \widetilde R(\Pi)$.
The map $CCS$ can be interpreted as an obstacle to $\tau^{hol}(\omega_{hol})$ inducing a 
trivial representation of $\Pi$, in fact in this case the Cheeger-Chern-Simons character 
reduces to an integer, the rank.

\subsection{Any Hilbert net bundle is an index}
\label{B3}

We now show that any Hilbert net bundle with finite rank 
(or, to be precise, its holonomy) can be interpreted as the index of a Fredholm module with coefficients in a net of $\rC^*$-algebras.
Of course, it is well-known that any element of the representation ring $R(\Pi)$ can be realized as the
index of a $K$-homology cycle with coefficients in some $\rC^*$-dynamical system, nevertheless
we feel that it is instructive to give a construction in terms of Hilbert net bundles.
\smallskip

We take a poset $K$ and define $\Pi := \pi_1^a(K)$. 
Given a separable Hilbert space $H_0$ and $d \in \bN$ we define 
$H := H_0 \otimes \bC^d$ and make the identification $H \simeq \oplus_{m \in \bN} \bC^d$.
Note that $\bU(d)$ acts on $H$ by means of the map $V \mapsto 1_0 \otimes V$
where $1_0 \in B(H_0)$ is the identity.
Consider the \emph{shift} operator $S^d \in B(H) $ defined as
\begin{equation}
\label{eq.shift0}
S^d(w) := 0 \oplus w_1 \oplus w_2 \oplus \cdots \ , \qquad  w = (w_m) \in  \oplus_{m \in \bN} \bC^d  \simeq H  \ , 
\end{equation}
and observe that, by construction,
\begin{equation}
\label{eq.shift}
[ S^d , 1_0 \otimes V ] = 0
\ , \qquad  V \in\bU(d)
\ .
\end{equation}
We now take a unital $\rC^*$-algebra $A \subset B(H)$ satisfying the following properties
\begin{equation}
\label{eq.A}
\left\{
\begin{array}{ll}
S^d t - t S^d \in K(H) \ ,
\\
(1_0 \otimes V) \cdot t \cdot (1_0 \otimes V^*) \in A \ ,  
\end{array}
\right.
\qquad  t \in A \ , \ V \in\bU(d)
\end{equation}
(for example, we may take $A = \{ \bC 1_0 + K(H_0) \} \otimes 1_d$, where $1_d \in \bM_d$ is the identity).
Note that the second relation in (\ref{eq.A}) says that $\bU(d)$ acts on $A$ by *-automorphisms.\smallskip 

As observed in Section \ref{A1}, to define an arbitrary net bundle is equivalent to giving an action of the homotopy group on the standard fibre of the net bundle. 
So let us start by considering  an arbitrary finite dimensional representation
$u : \Pi \to \bU(d)$ and the corresponding finite rank Hilbert net bundle  $(\cE,u_*)_K$  defined by
$\cE_o \equiv \bC^d$ for any $o\in K$ and   $u_{*,o'o} := u(g_{o'o})$ for any inclusion $o \leq o'$,  
where $g_{o'o} \in \Pi$ is defined by  (\ref{rep:2a}). The ampliation of the inclusions maps
$u_{*,o'o}$ to the Hilbert space $H=H_0 \otimes \bC^d$ i.e.
\[
U_{o'o} := 1_0\otimes u_{*,o'o} \ , \qquad  o \leq o' \ , 
\]
yields a  Hilbert net bundle $(\cH,U)_K$ where   $\cH_o \equiv H$ for any $o\in K$. 
%
%Note that, by construction, we may regard $(\cH, U)_K$ as the tensor product of the trivial Hilbert net bundle  with fibre $H_0$ by the Hilbert net bundle $(\cE,u_*)_K$.  Furthermore, 
%
Passing to  the adjoint action $\ad U$, the global invariance of the $\rC^*$-algebra $A$ 
under the action of  $\bU(d)$  (\ref{eq.A}) gives the $\rC^*$-net bundle  
\[
(\cA,\jmath)_K \ \ , \ \ \cA_o :=A \ \ , \ \ \jmath_{o'o} := \ad U_{o'o}|_A
\ \ , \ \ 
\forall o \leq o' \in K
\ ,
\]
and the corresponding \emph{defining representation}
\[
\pi : (\cA,\jmath)_K \to (\cBH,\ad U)_K  \ \ , \ \  \pi_o(t)=t  \ , \qquad  o\in K , \ t\in A \ , 
\]
where $\cBH_o \equiv \cB(H)$ for any $o\in K$. 
We now define $\hat H_0 := H_0 \oplus H_0$ and the Hilbert net bundle $(\hat \cH,\hat U)_K$, 
where $\hat \cH_o := \hat H_0 \otimes \cE_o $ for any $o\in K$ and 
\[
\hat U_{o'o} := id_{\hat H_0} \otimes  u_{*, o'o} \ , \qquad   o \leq o' \ ,
\]
endowed with the grading
\[
\Gamma_o ((v \oplus v') \otimes w) := (v \oplus (-v')) \otimes w  \ ,  \qquad v \oplus v' \in \hat H_0,  \ w \in \cE_o \ .
\]
We have isomorphisms 
$\hat \cH_o = \hat H_0 \otimes \bC^d \simeq H \oplus H = \cH_o \oplus \cH_o$ 
preserved by the inclusion maps $U$, $\hat U$, so there is a decomposition 
$(\hat \cH,\hat U)_K  \simeq  (\cH,U)_K  \oplus (\cH,U)_K$ 
and a graded representation $\hat \pi$ of $(\cA,\jmath)_K$ (endowed with the trivial grading),
\[
\hat \pi_o(t) := \pi_o(t)\oplus \pi_o(t)
%\left(
%\begin{array}{cc}
%\pi_o(t) & 0 \\
%0          & \pi_o(t)
%\end{array}
%\right)
\ ,\qquad  t \in \cA_o
\ , \ o \in K
\ .
\]
To add a Fredholm module structure we set
\[
F_o := 
\left( 
\begin{array}{cc}
0 & S^d \\
S^{d,*} & 0
\end{array}
\right)
\ \in \cB^- \hat \cH_o
\ , \qquad  o \in K
\ .
\]
Using (\ref{eq.shift}), for any inclusion $o\leq o'$ we find 
$\hat U_{o'o} F_o = F_{o'} \hat U_{o'o}$,
so $F := \{ F_o \}$ is an odd section of $(\cB \hat \cH,\ad \hat U)_K$.
Moreover 
$F_o = F_o^*$,   
$(1 - F_o^2) \hat \cH_o \simeq \cE_o$ and  
$[ F_o , \pi_o(t) ] \in K(\hat \cH_o)$, 
for any $o\in K$ (the last relation follows by the first of (\ref{eq.A})),
so Def.\ref{def.fred} is fulfilled and $\omega := (\hat \pi,\hat U,F) \in \sF^0(\cA,\jmath)_K$.

We now show that the index $\tau^{hol}(\omega) \in \widetilde R(\Pi)$ can be interpreted as the holonomy representation 
$u : \Pi \to \bU(d)$ of $(\cE, u_*)_K$ that was arbitrarily chosen at the beginning. 
To this end we note that by (\ref{eq.fdim}) we have
\[
\tau^{hol}(\omega) \ = \ 
{\mathrm{index}}(F_*) \ = \  
[\ker S^{d,*}] - [ \ker S^d ] \ = \ 
[\ker S^{d,*}] \ ,
\]
where $[M]$ denotes the equivalence class of the $\Pi$-vector space $M$ in $\widetilde R(\Pi)$
(for example, see \cite[Ex.4.2.10]{Val}).
Now, by (\ref{eq.shift0}) we have that $\ker S^{d,*}$ is the first direct summand of $H \simeq \oplus_m \bC^d$,
that is, $\ker S^{d,*} \simeq \bC^d$. The $\Pi$-action on $H$ is given by
$1_0 \otimes u(g)$, $\forall g \in \Pi$,
which becomes $\oplus_m u(g)$ in the decomposition $H \simeq \oplus_m \bC^d$; 
thus the $\Pi$-action on $\ker S^{d,*}$ is given by $u$ as desired.

\

An interesting class of examples is given by spaces $X$ homotopic to the circle $S^1$,
as $S^1$ itself (conformal theories), anti-de Sitter space-times and de Sitter space-time
with order 2 (general relativity). There we have, for any commutative ring $\cR$,
\[
\Pi \simeq \bZ
\ \ , \ \
\widetilde R(\Pi) \simeq \bZ[\bT] := \left\{ \sum_i k_i z_i \, : \, k_i \in \bZ \, , \, z_i \in \bT  \right\}
\ \ , \ \
H^{odd}(X,\cR) = H^1(X,\cR) \simeq \cR \ .
\]
If $(\cL,U)_K$ is a rank one Hilbert net bundle with holonomy
$u : \bZ \to \bT$,
then (\ref{eq.ccs}) takes the form
\[
ccs(u) \ = \ 1 + [c^\uparrow_1(u)] \ \in \bZ \oplus H^1(X,{\mathbb{R/Q}}) \ .
\]
Now, it is easily seen that for any 
$z \in H^1(X,{\mathbb{R/Z}}) \simeq {\mathbb{R/Z}}$ 
there is $u$ such that $c^\uparrow_1(u) = z$ 
(take $u(k) := \exp(2 \pi i k z)$, $k \in \bZ$).
This implies that any $z \in H^1(X,{\mathbb{R/Z}})$
is the first Chern class of some rank one Hilbert net bundle.
Since $\bZ[\bT]$ is generated as an additive group by classes of rank one Hilbert net bundles 
(labelled by $\bT$), we conclude that any cohomology class in (\ref{eq.ccs.net}) arises from some Fredholm $(\cA,\jmath)_K$-module.

\subsection{Applications to quantum field theory: sectors in curved space-times}
\label{B4}

As we mentioned in the introduction the interest in the Chern character (\ref{eq.ccs.net}) arises from its physical interpretation. 
In the following lines we first recall the approach by Longo (\cite[\S 6]{Lon01}) which, 
postulating the existence of a supersymmetry, showed how the statistical dimension of a superselection sector 
can be interpreted as an index. 
Then, generalizing the discussion of \cite[\S 6]{RV1} to non-abelian fundamental groups,
we present a new construction that does not assume existence of supersymmetries 
and can be applied to sectors with non-trivial holonomy in the sense of \cite{BR09,Vas14}.
The object that we obtain is a Fredholm module whose index takes values in $\widetilde R(\Pi)$ and, 
as a consequence of (\ref{eq.ccs.net}), in the odd cohomology of the space-time.

\paragraph{Field nets.}
We start with a field net $(\cF,\jmath)_K$ over a good base $K$ generating the topology of the space-time $X$ under consideration 
(\emph{doublecones} in the case of the Minkowski space-time, \emph{diamonds} in more generic cases, see \cite{GLRV01,Ruz05,BR09}), 
realized in a "defining" Hilbert space representation
\begin{equation}
\label{eq.B4.1}
\pi : (\cF,\jmath)_K \to B(H) \ ,
\end{equation}
that is, $\cF_o \subset B(H)$ for any $o \in K$ and $\jmath_{o'o} : \cF_o \to \cF_{o'}$, $o \subseteq o'$, is the inclusion map.
The Hilbert space $H$ is assumed to be separable.
We define the global field algebra 
$\vec{\cF}$ as the $\rC^*$-algebra generated by the local algebras  $\cF_o$ for any $o \in K$,
and assume that it is irreducible, $\vec{\cF}''=B(H)$, and that a strongly compact group $G \subset U(H)$ of internal symmetries acts, 
\emph{via} the adjoint action, by automorphisms of
$(\cF,\jmath)_K$,
\[
\alpha_g(F) \ := \ gFg^* \in \cF_o \ \ , \ \ o \in K \, , \, F \in \cF_o \, , \, g \in G \, ,
\]
in such a way that the observable net
$\cA_o := \cF_o \cap G'$, $o \in K$,
is defined, with global algebra $\vec{\cA}$,  i.e.\, the  $\rC^*$-algebra generated by the local algebras  $\cA_o$ for any $o \in K$.
We have the decomposition 
\[
H \ = \ \bigoplus_{\si \in \check{G}} H_\si \otimes L_\si \ ,
\]
where $\check{G}$ is the set of irreducible (finite-dimensional) representations $\si : G \to U(L_\si)$ and $H_\si$ is the multeplicity space.
In such a decomposition, the restriction $\pi_0 := \pi |_{\vec{\cA}}$ takes the form
\begin{equation}
\label{eq.B4.2}
\pi_0 : (\cA,\jmath)_K \to B(H)
\ \ , \ \
\pi_0(T) = \oplus_\si \{ \pi_\si(T) \otimes 1_\si \}
\ \ , \ \ 
T \in \vec{\cA}
\ ,
\end{equation}
where $1_\si \in B(L_\si)$ is the identity and $\pi_\si : (\cA,\jmath)_K \to B(H_\si)$ is an irreducible Hilbert space representation (that is, a \emph{sector}).
As shown by Doplicher, Haag and Roberts \cite[\S 1]{DHR}, \cite[\S 3.2]{GLRV01}, \cite[Eq.3.8]{BR09}, in the cases of interest all the Hilbert spaces $H_\si$, $\si \in \check{G}$, are isomorphic to $H_\iota$, 
where $\iota$ is the trivial representation of $G$; in particular 
$\pi_\si$ turns out to be unitarily equivalent to $\pi_\iota$ when restricted to $(\cA,\jmath)_S$ for some subset $S$ of $K$ 
contained in the causal complement of a given $e \in K$.
Thus, up to unitary equivalence, we may rewrite the representations $\pi_\si$ in (\ref{eq.B4.2}) as
\begin{equation}
\label{eq.B4.2a}
\pi'_\si : (\cA,\jmath)_K \to B(H_\iota) \ .
\end{equation}
It is well-known that the above assumptions are fulfilled in the case of the Minkowski space-time.
The fact that this scenario is reasonable also in generic curved space-times with at least two spatial dimensions has been discussed, for example, 
in \cite[\S 3]{GLRV01}, \cite[\S 3.2]{BR09}.
Reference representations of the type (\ref{eq.B4.1}) arise from Hadamard states which play the r\^ole of the usual vacuum state in the Minkowski space.

\paragraph{Supercharges and index.}
In \cite[\S 6]{Lon01} some assumptions are made to illustrate how a supersymmetric structure yields an interpretation of the \emph{statistical dimension}
$\dim L_\si$
as an index. The first is that the superselection sectors are implemented by \emph{localized endomorphisms}, that is
$\pi_\si$ is unitarily equivalent to $\pi_\iota \circ \check{\si}$ with $\check{\si} \in {\bf end}\vec{\cA}$.
Assuming that a field reconstruction as in \cite{DR90} can be made, 
there are mutually orthogonal partial isometries $\psi^\si_1 , \ldots , \psi^\si_{\dim L_\si} \in \vec{\cF}$ with total support the identity,
such that
\[
\ad \psi^\si(t) \, := \, \sum_k \psi^\si_k t \psi^{\si,*}_k \, = \, \check{\si}(t)
\ \ , \ \
\forall t \in \vec{\cA}
\ ;
\]
note that $\ad \psi^\si$ defines a *-endomorphism of $B(H)$, and the previous expression says that it restricts to a *-endomorphism of $\vec{\cA}$.
The vector space spanned by $\psi^\si_1 , \ldots , \psi^\si_{\dim L_\si}$ has scalar product $\psi,\psi' \mapsto \psi^*\psi \in \bC 1$, 
is stable under the adjoint $G$-action and is isomorphic to $L_\si$ as a $G$-Hilbert space.
The second assumption is that $\cF$ is supersymmetric: we have a Hamiltonian operator $\textbf{H}$ defined on $H$
admitting an odd square root $D$, the {\em supersymmetry}.
A simple argument based on the McKean-Singer formula shows that if there is a unique vector generating the kernel of $\textbf{H}$ then
${\mathrm{index}} \, (D^+) =$ $\dim \ker \textbf{H} = 1$,
where $D^+ : H_+ \to H_-$ is the even/odd component of $D$. 
Assuming that $\si$ is bosonic (i.e., any $\psi^\si_k$ commute with a given grading operator $\gamma = \gamma^{-1} \in G \cap G'$),
and considering the operator $\ad \psi^\si(D^+)$, the argument of \cite[\S 6]{Lon01} yields 
${\mathrm{index}}(\ad \psi^\si(D^+)) = \dim L_\si$.

\paragraph{The statistical dimension as an $\widetilde R(\Pi)$-valued index.}
We now give a construction of nets of Fredholm modules that yields the statistical dimension as an index without assuming existence of supercharges.
Our motivations are given by the following facts, making the construction of the previous paragraph problematic in cases of interest:
\begin{itemize}
\item First, in a generic space-time $X$ localized endomorphisms of the type $\check{\si}$ can be defined only on suitable *-subalgebras of $\vec{\cA}$,
      see \cite{GLRV01,BR09};
\item Secondly, supersymmetries turn out to be very singular objects and, at least in conformal theory, do not yield finite-dimensional Fredholm modules, 
      see \cite[\S 7]{Lon01}, \cite{CHKL10,BG07} and \cite{Con,Con88,JLO88}.
      Thus the interpretation of the statistical dimension as an index is problematic for finite-dimensional sectors in curved space-times
      with at least two spatial dimensions.
\end{itemize}
We proceed as follows. We consider a finite direct sum of mutually inequivalent, irreducible representations
\[
\varsigma \, := \, \si_1 \oplus \ldots \oplus \si_n \ \ , \ \ \si_k \in \check{G} \ , \ k=1,\ldots,n
\]
acting on the Hilbert space 
$L_\varsigma := \oplus_k L_{\si_k}$.
In correspondence with any $\si_k$ there is a projection
$e_k \in M_\varsigma := B(L_\varsigma) \cap \varsigma(G)'$;
any $e_k$ is central, that is, it commutes with elements of $M_\varsigma$, see \cite[\S 8.3]{KirG}.

\smallskip

Now, using the isomorphism 
$\oplus_k (H_\iota \otimes L_{\si_k}) \simeq H_\iota \otimes L_\varsigma$,
we may regard
$\oplus_k^n \pi'_{\si_k} \otimes 1_{\si_k}$
as the representation
\begin{equation}
\label{eq.B5}
\pi_\varsigma
\, : \, 
(\cA,\jmath)_K \to B(H_\iota \otimes L_\varsigma ) 
\ \ , \ \ 
\pi_\varsigma(T) := \sum_k^n \pi'_{\si_k}(T) \otimes e_k
\ \ , \ \ 
T \in \cA_o
\, , \,
o \in K
\ .
\end{equation}
By elementary calculations we have
\begin{equation}
\label{eq.ex1}
[ Z , e_k ] \, = \, [ Z , \varsigma(g) ] \, = \, 0
\ \ , \ \ 
[ 1_\iota \otimes Z , \pi_\varsigma(T) ] \, = \, [ 1_\iota \otimes Z , t \otimes 1_\varsigma ] \, = \, 0 \ ,
\end{equation}
for all $Z \in U(M_\varsigma)$, $o \in K$, $T \in \cA_o$, $t \in B(H_\iota)$, $g \in G$, $k=1,\ldots,n$.

\smallskip

We now consider a unitary representation of $\Pi := \pi_1(K) \simeq \pi_1(X)$ of the type
\[
\varrho : \Pi \to U(M_\varsigma) \subseteq U(L_\varsigma) \ .
\]
The idea is that $\varrho$ corresponds to the "localized holonomy" of the sector that we are going to construct, 
see \cite[Theorem 4.1 and Theorem 6.4$(i)$]{BR09},
whilst $\pi_\varsigma$ plays the r\^ole of the charge component in the sense of \cite[\S 6.1]{BR09}.
As a consequence of (\ref{eq.ex1}) we have
$1_\iota \otimes \varrho(\Pi) \subseteq \pi_\varsigma(\vec{\cA})'$,
in accord with \cite[Theorem 6.4$(i)$]{BR09}.
We define the Hilbert net bundle $(\cH,U)_K$,
\[
\cH_o := H_\iota \otimes L_\varsigma 
\ \ , \ \ 
U_{o'o} := 1_\iota \otimes \varrho(g_{o'o})
\ \ , \ \ 
\forall o \subseteq o' \in K
\ ,
\]
where $g_{o'o} \in \Pi$ is defined by (\ref{rep:2a}). Then we set
\[
\pi_\varrho : (\cA,\jmath)_K \to (\cBH,\ad U)_K
\ \ , \ \ 
\pi_{\varrho,o}(T) := \pi_\varsigma(T)
\ \ , \ \
\forall o \in K \, , \, T \in \cA_o
\ ;
\]
this is a well-defined representation as, by (\ref{eq.ex1}),
\[
[ 1_\iota \otimes \varrho(g_{o'o}) , \pi_\varsigma(T) ] \ = \ 0
\ \ \Rightarrow \ \
U_{o'o} \pi_{\varrho,o}(T) U_{o'o}^* = \pi_{\varrho,o}(T) = \pi_{\varrho,o'}(T)
\ ,
\]
for all $T \in \cA_o \subseteq \cA_{o'}$, $o \subseteq o' \in K$.
Note that $\pi_\varrho$ is a \emph{net representation} in the sense \cite[\S 2]{BR09}: using the terminology of the above-cited reference, 
we call \emph{the topological dimension} of $\pi_\varrho$ the dimension of the finite von Neumann algebra
$\varrho(\Pi)'' \subseteq M_\varsigma$,
whilst ${\mathrm{dim}}L_\varsigma$ plays the r\^ole of the statistical dimension of $\pi_\varrho$.

\smallskip

Let us now consider the reference state $w \in H_\iota$ defining the representation $\hat{\pi}_\iota$ \emph{via} the GNS-construction.
Then we may take $w$ as the first element of a base for $H_\iota$ and consider the corresponding shift operator
$S_w \in B(H_\iota)$
such that $S_w^*S_w = 1_\iota$ and $S_w S_w^*$ is the projection onto $H_\iota \ominus \bC w$.
In concrete cases such a decomposition arises when $(\cA,\jmath)_K$ is the net generated by a Wightman field $\phi$:
for any test function $f$ supported in $o \in K$, the quantum field $\phi(f)$ appears as an unbounded operator affiliated to $\cF_o$,
and $\phi(f) w \in H_\iota \ominus \bC w$ in the hypothesis that there are no spontaneously broken symmetries
(i.e., the one-point Wightman function of $\phi$ vanishes, see \cite[\S 5.1]{Ste}).

\smallskip

Now, we define the Hilbert net bundle
$(\hat{\cH},\hat{U})_K$,
where
$\hat{\cH}_o := \cH_o \oplus \cH_o$,
$\hat{U}_{o'o} := U_{o'o} \oplus U_{o'o}$, $o \subseteq o'$,
and the graded representation
$\hat{\pi}_\varrho : (\cA,\jmath)_K \to (\cB\hat{\cH},\ad\hat{U})_K$,
$\hat{\pi}_\varrho := \pi_\varrho \oplus \pi_\varrho$.
Then we set
\[
F_o \in \cB\hat{\cH}_o
\ \ , \ \ 
F_o \, := \, 
\left(
\begin{array}{cc}
0 & S_w^* \otimes 1_\varsigma  \\
S_w \otimes 1_\varsigma & 0
\end{array}
\right)
\ \ , \ \ 
\forall o \in K
\ .
\]
We have $F_o=F_o^*$, $F_o^2-1_o \in K(\hat{\cH}_o)$, $\forall o \in K$, and, using (\ref{eq.ex1}),
$\hat{U}_{o'o}F_o = F_{o'}\hat{U}_{o'o}$, $\forall o \subseteq o'$.
On the line of \cite[\S 5.1]{HR}, we define the "relative" \emph{dual net}
\[
\cA^{\varrho,w}_o \, := \, \{ T \in \cA_o \, : \, [\hat{\pi}_{\varrho,o}(T) , F_o] \in K(\hat{\cH}_o)  \} 
\ \ , \ \ 
o \in K
\ ,
\]
so that, regarding $\hat{\pi}_\varrho$ as a representation of $(\cA^{\varrho,w},\jmath)_K$, we have the 
even Fredholm $(\cA^{\varrho,w},\jmath)_K$-module $(\hat{\pi}_\varrho,\hat{U},F)$.
It turns out
\[
\tau^{hol}(\hat{\pi}_\varrho,\hat{U},F) \, = \, 
{\mathrm{index}}(F_*) \, = \, 
[ \ker (S_w^* \otimes 1_\varsigma) ] \, = \, 
[\varrho] \in \widetilde R(\Pi)
\ .
\]
Applying the Cheeger-Chern-Simons character, we obtain
\[
CCS(\hat{\pi}_\varrho,\hat{U},F) \ = \
ccs([\varrho]) \ = \ 
{\mathrm{dim}}L_\varsigma + 
\sum_{k=1}^{ {\mathrm{dim}}L_\varsigma } 
\frac{ (-1)^{k-1}  }{ (k-1)! } \ [ c^\uparrow_k(\varrho) ]
\, 
\in \bN \oplus H^{odd}(X,{\mathbb{R/Q}})
\ .
\]

\section{Nets of spectral triples.}
\label{C}

Nets of spectral triples have been introduced in the setting of conformal field theory 
to study structural properties of nets of *-algebras admitting a supercharge structure (\cite{CHKL10}).
Existence of supercharges has been proved for models in low dimensional quantum field theory and this, indeed, 
has been a motivation for studying entire cyclic cohomology and its variants (\cite[Chap. 4.9.$\beta$]{Con}, \cite{Con88,JLO88}).

\smallskip

In the present section we show that, when considering representations with non-trivial holonomy, the holonomy dynamical system of 
the given net of *-algebras is, instead of the universal algebra, a natural candidate for a description of 
the associated nets of spectral triples in terms of equivariant spectral triples.
The following definition is a natural generalization of \cite[Def.3.9]{CHKL10}.
\begin{definition}
Let $(\cA,\jmath)_K$  be a net of *-algebras.
An \textbf{even net of spectral triples over $(\cA,\jmath)_K$}, 
denoted by $(\cH,U,\Gamma,D,\pi)$, is given by: 
\begin{itemize}
\item[(i)]   A graded Hilbert net bundle $(\cH,U;\Gamma)_K$ carrying a graded 
             representation $\pi$ of $(\cA,\jmath)_K$;
\item[(ii)]  A family $D = \{ D_o : {\mathrm{Dom}}(D_o) \to \cH_o \}$ of odd, unbounded 
             selfadjoint operators such that 
             \[
             \ad U_{o'o}(D_o) = D_{o'} \ ,
              \qquad  o \leq o' \ ,
              \]
             fulfilling the property of $\theta$-summability,
             ${\mathrm{Tr}} (e^{-\beta D_o^2}) < \infty$,
             $\forall \beta > 0$,
             and defining the superderivations
             \[
             \left\{
             \begin{array}{ll}
             \delta_o (t) := D_o t - \ad \Gamma_o (t) D_o \ , 
             \\
             \forall o \in K \ , \ 
             t \in {\mathrm{Dom}}(\delta_o) := 
             \{ t \in \cB(\cH_o) \ | \ \exists b \in \cB(\cH_o) : \ad \Gamma_o (t) D_o \subset D_ot -b   \}
             \ ,
             \end{array}
             \right.
             \]
             such that
             $\pi_o(\cA_o) \subseteq {\mathrm{Dom}}(\delta_o)$, $\forall o \in K$.
\end{itemize}
\end{definition}

\begin{remark}
Let $t \in {\mathrm{Dom}}(\delta_o)$. Then 
$\delta_o(t) :=  D_o t - \Gamma_o t \Gamma_o D_o \in \cB(\cH_o)$
and, defining $t' := \ad U_{o'o}(t)$ we find
\[
\delta_{o'} (t') =
D_{o'} t' - \Gamma_{o'} t' \Gamma_{o'} D_{o'} =
\ad U_{o'o} ( D_o t - \Gamma_o t \Gamma_o D_o ) =
\ad U_{o'o} \circ \delta_o(t)
\ .
\]
This implies that $t' \in {\mathrm{Dom}}(\delta_{o'})$, so the family
$\{ {\mathrm{Dom}}(\delta_o) \}$ is stable under the inclusion map $\ad U$.
\end{remark}

We add further structure defining arrows. A unitary map between nets of spectral triples
$(\cH,U,\Gamma,D,\pi)$, $(\cH',U',\Gamma',D',\pi')$ is a \emph{unitary} morphism
\[
V : (\cH,U) \to (\cH',U')
\ \ , \ \
V_{o'} \circ U_{o'o} = U'_{o'o} \circ V_o
\ \ , \ \
o \leq o'
\ ,
\]
such that, for each $o \in K$,
\[
V_o \circ \Gamma_o = \Gamma'_o \circ V_o
\ \ , \ \ 
\ad V_o \circ \pi_o  = \pi'_o
\ \ , \ \
\ad V_o \circ D_o = D'_o
\ .
\]
In this case we write
\[
V : (\cH,U,\Gamma,D,\pi) \to (\cH',U',\Gamma',D',\pi') \ .
\]
We denote the category of nets of spectral triples over $(\cA,\jmath)_K$ by
${\mathrm{ncg}}(\cA,\jmath)_K$.

\begin{definition}
Let $\Pi$ be a group acting on the *-algebra $A$. 
A \textbf{$\Pi$-equivariant (even) spectral triple over $(A,\Pi)$}, denoted by $(H,\Gamma,D,\pi,u)$, 
is given by: 
(i)          A graded Hilbert space $(H;\Gamma)$ with a graded, unitary representation
             $u : \Pi \to U(H)$, $u_g \Gamma = \Gamma u_g$, $\forall g \in \Pi$;
(ii)         A graded, $\Pi$-equivariant representation
             \[
             \pi : A \to \cB(H)
             \ \ : \ \ 
             \pi(gt) = u_g \pi(t) u_g^*
             \ , \
             t \in A
             \ , \
             g \in \Pi
             \ ;
             \]
(iii)        An odd, unbounded, selfadjoint $\Pi$-invariant operator 
             \[
             D  : {\mathrm{Dom}}(D) \to H
             \ \ : \ \
             u_gD = Du_g
             \ , \
              g \in \Pi
             \ ,
             \]
             fulfilling the property of $\theta$-summability,
             ${\mathrm{Tr}} (e^{-\beta D^2}) < \infty$,
             $\forall \beta > 0$,
             and defining the superderivation
             $\delta := D \cdot - \ad \Gamma(\cdot) D$
             such that
             $\pi(A) \subseteq {\mathrm{Dom}}(\delta)$.
\end{definition}

Let $S := (H,\Gamma,D,\pi,u)$ and $S' := (H',\Gamma',D',\pi',u')$ be $\Pi$-equivariant
spectral triples. An arrow from $S$ to $S'$ is a $\Pi$-equivariant unitary 
\[
V : H \to H'
\ \ , \ \
V \circ u_g = u'_g \circ V
\ \ , \ \ g \in \Pi
\ ,
\]
such that
$V \circ \Gamma = \Gamma' \circ V$,
$\ad V \circ \pi  = \pi'$,
$\ad V \circ D = D'$.
We denote the category of $\Pi$-equivariant spectral triples over $A$ by
${\mathrm{ncg}}_\Pi(A)$.\smallskip

Let $(\cA,\jmath)_K$ be a net of *-algebras. 
The construction of the enveloping net bundle applies without substantial modifications in this case,
so since in the sequel we shall make explicit computations we recall the construction of its standard fibre,
which is a *-algebra  carrying a $\pi_1^a(K)$-action by *-automorphisms.
We pick $a \in K$ and define $A$ as the *-algebra of pairs 
$(p,t)$, $p : o \to a$, $\forall o \in K$, $t \in \cA_o$,
with relations 
\begin{equation}
\label{def.spec3.1}
\left\{
\begin{array}{ll}
(p,t) + \lambda (p,t') = (p,t+\lambda t') \ \ , \ \ 
(p,t) \cdot (p,t') = (p,tt') \ \ , \ \
(p,t)^* = (p,t^*)
\\
(p*(\tilde oo),t) = (p,\jmath_{\tilde oo}(t))
\ \ , \ \
(p,t) = (p',t) \ \ {\mathrm{if}} \ \ p \sim p'
\end{array}
\right.
\end{equation}
(the symbol $\sim$ stands for homotopy equivalence of paths in the sense of \S \ref{A0}).
Defining $\Pi := \pi_1^a(K)$ yields the $\Pi$-action
\begin{equation}
\label{eq.spec3.1}
\Pi \to \textbf{aut} A
\ \ , \ \
g(p,t) := (g*p,t)  \ , \qquad  g : a \to a \ ,
\end{equation}
and reasoning as in \S \ref{A2} we conclude that 
$\Pi$-equivariant representations $(\pi,u)$ of $A$ are in
one-to-one correspondence with representations of $(\cA,\jmath)_K$.
By construction, $A$ is the fibre of the enveloping net bundle of  $(\cA,\jmath)_K$.

\begin{theorem}
\label{thm.spec3}
Let $K$ be a poset with fundamental group $\Pi$ and $(\cA,\jmath)_K$ a net of *-algebras. 
Then there is an equivalence
${\mathrm{ncg}}(\cA,\jmath)_K \leftrightarrow {\mathrm{ncg}}_\Pi(A)$.
\end{theorem}

\begin{proof}
\emph{Step 1}. We give an explicit description of the equivalence associating unitary representations of the fundamental group to Hilbert net bundles. 
Given a pair $(H,u)$ where $u:G\to U(H)$ is a unitary representation on the Hilbert space $H$, following  Section \ref{A1} the corresponding Hilbert net bundle $(\cH,U)_K$ is defined as follows. 
Let us pick $a\in K$, so  $G=\pi^a_1(K)$. Take a path-frame $P_a$ over $a$ and define 
$\cH_o\equiv H$ for any $o\in K$, as usual, whilst  the inclusion maps $u_*$ are defined as 
\[
U_{o'o}:= u_{[p_{ao'}*(o';o'o)*p_{oa}]} \ , \qquad  o\leq o'  \ , \qquad \qquad (*)
\] 
where the the square bracket denotes the homotopy equivalence class of of the loop.  
Conversely, if  $(\cH,U)_K$ is a Hilbert net bundle then we set  $H:=\cH_a$. 
Moreover, we extend the inclusion maps to 1-simplices  $b  \in \Si_1(K)$ by setting 
$U_b:=  U_{\partial_0b,|b|}\cdot U_{|b|,\partial_1b}$,
and then to paths $p=b_n*\cdots *b_2*b_1$ by setting 
$U_p:=U_{b_n}\cdots U_{b_2}\cdot U_{b_1}$. 
Finally, since any element of $\Pi$ is of the form $[p]$ for some loop $p$ over $a$ 
we get the representation
\[
U_* : \Pi \to U(H)
\ \ , \ \
U_{*,g} := U_p\ , \qquad  g=[p] \in \Pi \ . \qquad \qquad (**)
\]
For details we address the reader to \cite{RV12}, where it is proved that 
the correspondence  
$(H,u) \mapsto (\cH,U)_K$
yields the desired equivalence between 
the category of unitary representations of $\Pi$ and the one of Hilbert net bundles on $K$.
To be concise, in the following steps we proceed \emph{as if we had an isomorphism} of categories,
bypassing all the computations needed to pass from $(\cH,U)_K$ to the 
Hilbert net bundle -- isomorphic to $(\cH,U)_K$ -- \ defined by applying (*) to $(H,U_*)$.

\emph{Step 2}.  $\Pi$-invariant (not necessarily bounded) operators of $H$ 
are in one-to-one correspondence with sections of the net bundle $(\cBH,\ad U)_K$ defined
by $(\cH,U)_K$.
For, take a $\Pi$-invariant operator $T : {\mathrm{Dom}}(T) \to H$, i.e.\
$u_g T u_g^* = T$, $\forall g \in \Pi$;
note that this implies that $u_g ( {\mathrm{Dom}}(T) ) = {\mathrm{Dom}}(T)$, $g \in \Pi$. 
We define the family of operators
\[
{\mathrm{Dom}}(T_o) := {\mathrm{Dom}}(T)
\ \ , \ \
T_o := T
\ \ , \ \
o \in K
\ ,
\]
and, by definition $(*)$ and $\Pi$-invariance,  find  that $U_{o'o}({\mathrm{Dom}}(T_o)) =\mathrm{Dom}(T_{o'})$  and    
$\ad U_{o'o} \circ T_o =T_{o'}$ for any inclusion $o\leq o'$. In particular, if $\Gamma : H \to H$ is bounded and $\Pi$-invariant  we have that 
$\ad U_{o'o} \circ \Gamma_o = \Gamma_{o'}$ for any inclusion $o \leq o'$ where
$\Gamma_o := \Gamma$ for any $o\in K$. 
Conversely, if $\{ T_o : {\mathrm{Dom}}(T_o) \to \cH_o \}$ is a family of operators
such that
\[
U_{o'o} {\mathrm{Dom}}(T_o) = {\mathrm{Dom}}(T_{o'})
\ \ , \ \
\ad U_{o'o} \circ T_o = T_{o'}
\ \ , \ \
o \leq o'
\ ,
\]
then defining
${\mathrm{Dom}}(t) := {\mathrm{Dom}}(T_a)$,
$t := T_a$
and  using the relation $(**)$ we get a $\Pi$-invariant operator, since
$u_g( {\mathrm{Dom}}(t) ) = \mathrm{Dom}(t)$ and
$\ad u_g \circ t = t$ for any $g\in \Pi$. \smallskip

\emph{Step 3}. It is clear that 
${\mathrm{Tr}} (e^{-\beta D^2}) =
 {\mathrm{Tr}} ( W e^{-\beta D^2} W^* ) =
 {\mathrm{Tr}} (e^{-\beta WD^2W^*})$
is invariant under the action of a unitary operator $W \in B(H)$, so $D$ is 
$\theta$-summable on $H$ if and only if the associated family $\{ D_o \}$ has 
$\theta$-summable elements on the Hilbert spaces $\cH_o$, $o \in K$. \smallskip

\emph{Step 4}. Let $\Gamma$ and $D$ as in the definition of $\Pi$-equivariant spectral triple.
Then the algebraic relations between $\Gamma$, $D$ are preserved when defining
the families $\{ \Gamma_o \}$, $\{ D_o \}$, thus 
\[
\Gamma^* = \Gamma = \Gamma^{-1}
\ \ , \ \
\ad \Gamma \circ \pi(A) = \pi(A)
\ \ , \ \
\Gamma D = - D \Gamma
\ \ , \ \
\Gamma u_g = u_g \Gamma
\]
if, and only if, 
\[
\Gamma_o^* = \Gamma_o = \Gamma_o^{-1}
\ , \
\ad \Gamma_o \circ \pi_o(\cA_o) = \pi_o(\cA_o)
\ , \
\Gamma_o D_o = - D_o \Gamma_o
\ , \ 
\Gamma_{o'} U_{o'o} = U_{o'o} \Gamma_o
\ .
\]

\emph{Step 5}. Finally we consider the domains of the involved superderivations.
Let $\pi : A \to B(H)$ be the representation associated to the given 
$\Pi$-equivariant spectral triple. Then for any $t \in \cA_o$, $o \in K$, 
by definition of $\pi_*$ we find
\[
\delta_o \circ \pi_{*,o}(t) \ = \ \delta \circ \pi(p_{oa},t) \in B(H)
\]
and $\pi_{*,o}(\cA_o) \subseteq {\mathrm{Dom}}(\delta_o)$.
Conversely, if $\eta : (\cA,\jmath)_K \to (\cBH, \ad U)_K$ is the graded
representation associated to the given net of spectral triples, then for each
$(p,t) \in A$ (i.e. $p : o \to a$ and $t \in \cA_o$) we find
\[
\begin{array}{ll}
\delta \circ \eta_*(p,t) & :=
D_a \cdot \ad U_p \circ \eta_o (t) - \ad \Gamma_a \circ \ad U_p \circ \eta_o(t) \cdot D_a = \\ & =
D_a \cdot \ad U_p \circ \eta_o (t) - \ad U_p \circ \ad \Gamma_o \circ \eta_o(t) \cdot D_a = \\ & =
\ad U_p ( D_o \cdot \eta_o(t) - \ad \Gamma_o \circ \eta_o(t) \cdot D_o ) = \\ & =
\ad U_p \circ \delta_o (t)
\ .
\end{array}
\]
Since this last operator belongs to $\ad U_p(\cB(\cH_o)) = \cB(\cH_a)$, we conclude that
$\delta \circ \eta_*(p,t) \in \cB(\cH_a)$, i.e. $\eta_*(p,t) \in {\mathrm{Dom}}(\delta)$.
\end{proof}

\

We conclude with a remark on the index of nets of spectral triples.
In conformal field theory nets of spectral triples $\omega := (\cH,U,\Gamma,D,\pi)$ arise from positive 
energy representations of Virasoro algebras (\cite{CHKL10}), and (again) $D$ is an odd square root of the Hamiltonian.
If we apply Theorem \ref{thm.spec3} {\em forgetting} the $\Pi$-action, then assigning the index we get a 
class $\omega_{ncg} \in HC_\eps(A)$, where $HC_\eps(A)$ is the entire cyclic cohomology (see \cite{Con88,JLO88}). 
This yields the map
\[
{\mathrm{index}}^{ncg} : K_0(A) \times HC_\eps(A)  \to \bC \ .
\]
The above map does not take the $\Pi$-action into account
thus, when the holonomy is not trivial, a more suitable invariant should be given by the
equivariant entire cyclic cohomology introduced in \cite{KL91,KKL91} which, nevertheless, 
at the authors knowledge is defined only for $\Pi$ finite.
In this case ${\mathrm{index}}^{ncg}$ takes values in the ring of central functions of $\Pi$ 
(that is, the centre of the convolution algebra $L^1(\Pi,\bC)$).

\appendix

\section{Locally constant bundles and secondary characteristic classes.}
\label{a.LC}

A bundle $\sB \to X$ is said to be \emph{locally constant} whenever it has an atlas such that 
the induced transition maps are locally constant. 
A bundle morphism is said to be locally constant whenever it appears, on each local chart,
as a locally constant map with values in the space of morphisms from a standard fibre to the other.
In general a locally trivial bundle is not locally constant and the same is true for a bundle morphism,
so locally constant bundles define a non-full subcategory of the one of locally trivial bundles on $X$.

The category of locally constant bundles can be equivalently described in terms of the holonomy representation,
as follows.
As a first step  we note that the universal cover $q : \tilde X \to X$ defines a right principal $\pi_1(X)$-bundle, so each fibre $q^{-1}(x)$, $x \in X$, is a right $\pi_1(X)$-space.
If $B$ is a left $\pi_1(X)$-space then it is easily verified that the space
\[
\sB := \tilde X \times_{\pi_1(X)} B 
\ := \
\{ \ [y,v] := \{ (yp,pv) \in \tilde X \times B \}_{p \in \pi_1(X)} \ \}
\]
is a locally constant bundle when endowed with the obvious projection 
$\sB \to X$.
A locally constant morphism is then defined by means of an intertwiner of holonomy representations.
We shall use the notion of locally constant bundle in the case in which the fibres are $\rC^*$-algebras
and Hilbert spaces, often finite-dimensional; for details in this last case (\emph{flat} bundles) 
see \cite[\S I.2]{Kob}.

\begin{remark}
\label{rem.b0}
Using the isomorphism $\pi_1^a(K) \simeq \pi_1(X)$ and Theorem \ref{thm.netdyn} we get the 
following functor from the category of Hilbert net bundles to the one of locally constant Hilbert bundles:
\[
(\cH,U)_K \ \mapsto \ \sH := \tilde X \times_{\pi_1(X)} \cH_*  \ ,
\]
where $\cH_* := \cH_a$ is a left $\pi_1(X)$-space by means of the holonomy $U_* : \pi_1^a(K) \to U \cH_*$.
This is indeed an equivalence, which also holds for $\rC^*$-net bundles.
\end{remark}

Although ordinary characteristic classes vanish on locally constant bundles when the base space is a manifold
(\cite[\S II.3]{Kob}), there is anyway a good substitute, the {\em Cheeger-Chern-Simons character}
\[
ccs : \widetilde R(\pi_1(X)) \to \bZ \oplus H^{odd}(X,{\mathbb{R/Q}})
\]
(see \cite[Theorem 8.22]{CS85}); here, for any commutative ring $\cR$,
\[
H^{odd}(X,\cR) \ := \ \oplus_k H^{2k+1}(X,\cR)
\]
is the odd singular cohomology 
and $\widetilde R(\Pi)$, for any group $\Pi$, is the Grothendieck ring generated by 
classes of finite-dimensional unitary representations.
$ccs$ is constructed in the following way: for \emph{any} rank $d$ vector bundle $\sE \to X$,
the Chern classes are given by closed forms
\[
c_k(\sE) \in Z^{2k}_{deRham}(X,\bR)
\ \ , \ \
k = 1 , \ldots , d
\ ,
\]
and there are cochains $c^\uparrow_k(\sE) \in C^{2k-1}(X,{\mathbb{R/Z}})$
such that
\begin{equation}
\label{eq.ccs0}
\left \langle c^\uparrow_k(\sE) , \partial \ell \right \rangle  
\ = \ 
\int_\ell c_k(\sE) \ {\mathrm{mod}} \bZ 
\ \ , \ \
\forall \ell \in C_{2k}(X)
\ ,
\end{equation}
where $C_{2k}(X)$ is the set of singular $2k$-chains and
$\partial : C_{2k}(X) \to Z_{2k-1}(X)$
is the boundary. Now, any $u \in \widetilde R(\pi_1(X))$ defines the locally constant
vector bundle 
$\sE := \wa X \times_{\pi_1(X)} \bC^d$ 
(see Rem.\ref{rem.b0}),
whose Chern classes are zero by \cite[Prop.3.1]{Kob}. By (\ref{eq.ccs0}) we have that
$c^\uparrow_k(\sE)$ vanishes on $\partial C_{2k}(X)$, so we have cocycles
\[
c^\uparrow_k(u) := c^\uparrow_k(\sE) \in Z^{2k-1}(X,{\mathbb{R/Z}}) 
\ \ , \ \
\forall k = 1 , \ldots , d \ ,
\]
and the desired character is defined by
\begin{equation}
\label{eq.ccs}
ccs (u) := d + \sum_{k=1}^d \frac{ (-1)^{k-1}  }{ (k-1)! } \ [ c^\uparrow_k(u) ]
\ \in \bZ \oplus H^{odd}(X,{\mathbb{R/Q}})
\ ,
\end{equation}
where $[ \cdot ]$ denotes the cohomology class mod ${\mathbb{Q}}$.
$ccs$ reduces to the rank on trivial representations of $\pi_1(X)$ (i.e., trivial locally constant bundles) 
and is additive under direct sums. 
Instead, for tensor products we have 
\[
ccs( u \otimes u') = d \, ccs(u') + d' \, ccs(u) - dd' \ .
\]

\section{Basics of (representable) $KK$-theory.}
\label{a.KK}

For reader's convenience, in this section we recall some notions of $KK$-theory 
(see \cite{Bla} and \cite[\S 2]{Kas88} for details). To be concise, we will call $\Pi$-$\rC^*$-algebra 
a $\rC^*$-dynamical system $(A,\Pi,\alpha)$ and write $\alpha_g(t) = gt$, $g \in \Pi$, $t \in A$.
In the same way, a $\Pi$-morphism is a $\Pi$-equivariant morphism between $\Pi$-$\rC^*$-algebras. 
A grading on $A$ is given by an automorphism $\gamma_A$ of $A$ with period $2$
commuting with the $\Pi$-action; a $\Pi$-morphism is said to be graded whenever 
it is a $\bZ_2$-morphism.

Given the graded $\Pi$-$\rC^*$-algebra $B$, a $\Pi$-Hilbert $B$-module $H$ is said
to be graded whenever, for all $g \in \Pi$, $v,w \in H$, $b \in B$, 
it turns out that:
(i) there is a linear map $\Gamma : H \to H$ such that
\[
\Gamma (vb) = \Gamma v \ \gamma_B(b)
\ \ , \ \
(\Gamma v , \Gamma w) = \gamma_B(v,w) \in B \ ;
\]
(ii) $\Pi$ acts on $H$ by isometric, invertible operators commuting with $\Gamma$, 
such that
\[
g(vb) = (gv)(gb)
\ \ , \ \
g(v,w) = (gv,gw) \in B
\ .
\]
Let now $\sA,\sB$ be graded $\Pi$-$C_0(X)$-algebras, eventually endowed with the trivial 
grading. A even/odd $\Pi$-{\em Kasparov} $\sA$-$\sB$-bimodule, denoted by $(\eta,\phi)$, 
is given by:
(i) a graded/ungraded $\Pi$-Hilbert $\sB$-bimodule $H$ carrying a graded/ungraded $\Pi$-representation
$\eta : \sA \to B(H)$
(here $B(H)$ is the $\Pi$-algebra of adjointable, right $\sB$-module operators), 
such that
\[
\eta(fa)vb = \eta(a)v(fb)
\ \ , \ \
\forall v \in H
\ , \ 
a \in \sA
\ , \ 
b \in \sB
\ , \
f \in C_0(X)
\ .
\]
(ii) a $\Pi$-continuous 
{\footnote{With this we mean that the map $\Pi \to B(H)$, $g \mapsto g \phi$, is norm-continuous.
           Clearly this is always true for discrete groups.}} 
odd/ungraded operator $\phi \in B(H)$ such that
\begin{equation}
\label{eq.KK.khom}
\left\{
\begin{array}{ll}
(\phi - \phi^*)\eta(t) 
\ , \
(\phi^2 - 1)\eta(t)
\ , \
[\phi,\eta(t)] 
\ , \
(g \phi-\phi)\eta(t) \ \in K(H)
\ ,
\\
\forall t \in \sA
\ , \
g \in \Pi
\ ,
\end{array}
\right.
\end{equation}
where $K(H) \subseteq B(H)$ is the ideal of compact $\sB$-module operators.
We denote the set of even/odd $\Pi$-{\em Kasparov} $\sA$-$\sB$-bimodules by $E_\Pi^*(\sA,\sB)$,
where $*=0$ in the even case and $*=1$ in the odd case.
When $\sB = C_0(X)$, $H$ is, in essence, a (separable) continuous field of Hilbert spaces.

We say that $(\eta_0,\phi_0) , (\eta_1,\phi_1) \in E_\Pi^*(\sA,\sB)$ are homotopic
whenever there is $(\eta,\phi) \in E_\Pi^*(\sA,\sB \otimes C([0,1]))$ such that
$(\eta_0,\phi_0) , (\eta_1,\phi_1)$ are obtained by applying to $(\eta,\phi)$ 
the evaluation morphism 
$\sB \otimes C([0,1]) \to \sB$
over $0,1 \in [0,1]$ respectively.

The representable $KK$-theory $RKK_\Pi^*(X;\sA,\sB)$ is defined as the abelian group
of homotopy classes of $\Pi$-Kasparov $\sA$-$\sB$-bimodules w.r.t. the
operation of direct sum. When $\Pi$ is trivial it will be omitted in the notation.
In particular, we define
\[
RK^*_\Pi(\sA) := RKK_\Pi^*(X;\sA,C_0(X)) 
\ \ , \ \
RK_*^\Pi(\sA) := RKK_\Pi^*(X;C_0(X),\sA) 
\ , \
*=0,1
\ ;
\]
these groups are called the {\em representable $K$-homology of $\sA$}
and, respectively, the {\em representable $K$-theory of $\sA$}. 
By \cite[Prop.2.21]{Kas88}, the Kasparov product induces the pairing
\begin{equation}
\label{index}
\left \langle \cdot , \cdot \right \rangle :
RK_*^\Pi(\sA) \times RK^*_\Pi(\sA) \to RK^0_\Pi(X) := RK_0^\Pi(C(X)) \ ,
\end{equation}
in essence the map defined by the index of continuous families of Fredholm operators.
When $X$ is compact $RK^0_\Pi(X)$ is the equivariant 
topological $K$-theory (see \cite[Prop.2.20]{Kas88}).

When $X$ reduces to a point we deal with $\rC^*$-algebras $A,B$ instead of $C_0(X)$-algebras 
$\sA,\sB$ (in particular, $H$ is simply a Hilbert space when $B = \bC$)
and obtain the usual notion of Kasparov bimodule, so we use the standard notations
\[
KK_\Pi^*(A,B) \ \ , \ \ K^*_\Pi(A) := KK_\Pi^*(A,\bC) \ \ , \ \ K_*^\Pi(A) := KK_\Pi^*(\bC,A) \ .
\]
Considering Kasparov $A$-$\bC$-bimodules $(\eta,\phi)$ such that $\phi$ is invariant,
i.e. $g \phi = \phi$ for all $g \in \Pi$, yields the {\em strongly equivariant $K$-homology} $\widetilde K^*_\Pi(A)$. 
When $\Pi$ is compact, averaging $\phi$ with respect to the Haar measure we obtain 
the invariant operator $\ovl \phi$ and this yields an isomorphism
$K^*_\Pi(A) \simeq \widetilde K^*_\Pi(A)$
(see \cite[V.11]{Bla}).
%
%In particular, $K^*_\Pi(A) \simeq \widetilde K^*_\Pi(A)$ for every finite group $\Pi$.
For $\Pi$ non-compact, the canonical map
\begin{equation}
\label{eq.bohr}
\widetilde K^*_\Pi(A) \, \to \, K^*_\Pi(A)
\end{equation}
may be not injective, as the homotopy equivalence relation on Kasparov bimodules is restricted to invariant operators.
%
%The above map may be understood in terms of group functoriality of $KK$.
%For, we consider the Bohr compactification $\beta\Pi$ defined as the Tannaka-Krein dual of the category
%of \emph{finite dimensional} unitary representations of $\Pi$; there is a canonical map 
%%
%\begin{equation}
%\label{eq.bohr2}
%\Pi \to \beta\Pi
%\end{equation}
%%
%arising as the dual map of the inclusion of the corresponding categories of unitary representations.
%By definition of $\beta\Pi$ there is a canonical isomorphism
%%
%$\widetilde K^*_\Pi(A) \simeq K^*_{\beta\Pi}(A)$,
%%
%thus the map (\ref{eq.bohr}) may be regarded as the one induced by (\ref{eq.bohr2}).

\smallskip

For $A=B=\bC$ we define $R(\Pi) := KK^0_\Pi(\bC,\bC)$.
When $\Pi$ is compact $R(\Pi)$ is the usual representation ring $\widetilde R(\Pi)$.
When $\Pi$ is discrete there is an isomorphism $R(\Pi) \simeq K^0(C^*\Pi)$, so (\ref{index}) takes the form
\begin{equation}
\label{Gindex}
\left \langle \cdot , \cdot \right \rangle :
K_*^\Pi(A) \times K^*_\Pi(A) \to R(\Pi) \ .
\end{equation}

%%*****************************************************************
%%*****************************************************************

\end{document}